\documentclass[12pt,leqno]{amsart}
\usepackage{a4wide}
\usepackage{amssymb}
\usepackage{amsmath}
\usepackage{amsthm}
\usepackage{verbatim}
\usepackage{txfonts}
\usepackage{fullpage}
\usepackage[all,cmtip]{xy} 
\usepackage{yfonts}
\usepackage[final]{changes}

\usepackage[T1]{fontenc}
\usepackage[utf8]{inputenc}
\usepackage{mathtools}   % loads �amsmath�
\usepackage{amsfonts}
\usepackage{mathrsfs, todonotes}

\numberwithin{equation}{section}

\theoremstyle{plain}

\newcommand{\A}{\ensuremath{{\mathbb{A}}}}
\newcommand{\C}{\ensuremath{{\mathbb{C}}}}

\newcommand{\Q}{\ensuremath{{\mathbb{Q}}}}

\newcommand{\F}{\ensuremath{{\mathbb{F}}}}

\newcommand{\E}{\ensuremath{{\mathbb{E}}}}

\newcommand{\B}{\ensuremath{{\mathbb{B}}}}

\newcommand{\charf}{\textbf{1}}
\newcommand{\Vol}{\text{Vol}}
\newcommand{\GL}{\ensuremath{{\text{GL}}}}

\newtheorem{theo}{Theorem}[section]
\newtheorem{lem}[theo]{Lemma}
\newtheorem{prop}[theo]{Proposition}
\newtheorem{cor}[theo]{Corollary}

\theoremstyle{remark}
\newtheorem{rem}[theo]{Remark}

\theoremstyle{definition}
\newtheorem{defn}[theo]{Definition}

\newtheorem*{cor*}{Corollary}

\newcommand{\zxz}[4]{\begin{pmatrix} #1 & #2 \\ #3 & #4 \end{pmatrix}}

\newcommand{\Hom}{\operatorname{Hom}}
\newcommand{\Ind}{\operatorname{Ind}}

\newcommand{\quadchar}{(\frac{\cdot}{q})}
\newcommand{\quadc}[1]{\left(\frac{#1}{q}\right)}
\newcommand{\Tr}{\text{Tr}}

\title{Test vectors for Waldspurger's period integral and application to the mass equidistribution on nonsplit torus}

\begin{document}
\author{Yueke Hu}

\address{MPIM}
\email{huyueke2012@gmail.com}

\begin{abstract}
In this paper we provide local test vector for Waldspurger's period integral, when the level of the representation $\pi_v$ is sufficiently large compared to the level of the character $\Omega_v$ over quadratic extension, while allowing joint ramifications. The test vectors we shall use are variants of classical newforms, and the size of the resulting local integral is asymptotically the inverse of convexity bound for $L(\Pi\otimes\Omega,1/2)$. Such test vectors are used to recover Gross-Prasad type test vectors. We also get vanishing result for local integral when using other test vectors. This phenomenon is used to prove the mass equidistribution of cuspidal newforms on nonsplit torus in depth aspect.
\end{abstract}
\maketitle

\section{introduction}
In \cite{Walds} Waldspurger studied the following period integral
\begin{equation}
\int\limits_{Z_\A\E^*\backslash \A_\E^*}F(e)\Omega(e)de.
\end{equation}
where $F\in \pi$ is a cusp form on $\GL_2$ and $\Omega $ is a character over a nonsplit quadratic extension $\E/\F$ such that $w_\pi\cdot\Omega|_{\A_\F^*}=1$. This integral provides an element in the space $\Hom_{\E^*}(\pi\otimes\Omega,\C)$, which might be trivial according to an epsilon value test by the work of Tunnell and Saito in \cite{Tu83} \cite{Sa}.

In \cite{Walds} Waldspurger established a formula relating this integral to $L(\Pi\otimes\Omega,1/2)$, the special value of twisted base change L-function for $\pi$.
Later on, explicit versions of the formula with level structures have been used to study arithmetic, equidistribution,  and subconvexity problems (see \cite{FW09} \cite{MW09} \cite{MV06} \cite{AP06} \cite{Zhang01}). Most of the work are based on the test vectors studied by Gross and Prasad in \cite{GP91}, where they assumed that $\pi$, $\E$ and $\Omega$ have disjoint ramifications. In a recent paper \cite{FMP13} File, Martin and Pitale gave the local test vector either when locally $\E_v$ is split over $\F_v$, or when $\E_v$ is a field and $c(\Omega_v)$ is sufficiently large compared to $c(\pi_v)$.

In this paper we are mainly interested in providing test vector for nontrivial element in $\Hom_{\E_v^*}(\pi_v\otimes\Omega_v,\C)$ and evaluating local integral in Waldspurger's formula
under the opposite condition for the field extension case, that is when $c(\pi_v)$ is sufficiently large compared to $c(\Omega_v)$, while allowing $\E_v/\F_v$ also to be ramified.

\subsection{Local integrals for different test vectors}
A particular interesting and challenging phenomenon in our setting is that when $\pi_v$ is a supercuspidal representation, the epsilon value test for $\Hom_{\E_v^*}(\pi_v\otimes\Omega_v,\C)$ could fail. For example when $\E_v/\F_v$ is inert and $c(\pi)>2c(\Omega)$, the space is nontrivial if and only if $c(\pi)$ is even. (See Lemma \ref{lem:epsiloninertsc} and \ref{lem:epsilonramifiedsc} for the story.)
One would expect test vectors and resulting local integrals to properly reflect such differences.

We will search for local test vectors coming from the following pool
\begin{equation}\label{eq:pooloftestvec}
\text{(diagonal translates of newform) or twisted newforms (see Definition \ref{def:twistednewform}). } \tag{*}
\end{equation}
One reason to start with such test vectors is that they are more nature from a historical point of view and can be potentially more useful for applications. Also when $\pi_v$ is a supercuspidal representation, elements from (*) form a basis of $\pi_v$, which guarantees that our search will be successful as long as the epsilon value test doesn't fail. 

The local results we get in this paper are actually threefold:
\begin{enumerate}
\item[(i)] Vanishing result for the local integral if we don't choose proper test vectors from the pool (*). See Proposition \ref{prop:vanishinglocalint}  in particular.
\item[(ii)] Nonvanishing results and size of the local integral for proper test vectors from (*). See Proposition \ref{proptestvecinertsc}, \ref{propofexplicitI}, \ref{prop:testvecramifiedsc} for supercuspidal representations and Proposition \ref{prop:inducedrepresentation}, \ref{propIindramified} for induced representations.
\item[(iii)] Using test vector from (ii), we obtain test vectors on which $\E_v^*$ acts by the character $\Omega_v^{-1}$. In addition such test vectors are invariant by principal congruence subgroup of depth roughly $c(\pi_v)/2$. See Corollary  \ref{corinertsc}, \ref{cor:ramifiedsc}, \ref{cor:inertInd} and \ref{cor:ramifiedInd}.
\end{enumerate}

The result in (i) may seem irrelevant for our purpose. But it greatly narrows our choice of test vectors from (*). It also show evidence (though not a proof) on why we will not find test vector when the epsilon value test fails in some cases. Lastly, vanishing and decaying results for local integrals are actually very important to prove power saving for global period integrals. We will give more details in the next subsection.

For the nonvanishing results in (ii), we shall see that when $c(\pi_v)\rightarrow \infty$, the size of the local integral for our test vectors from (*) is asymptotically 
$$ \frac{1}{q^{c(\pi_v)/2}},$$ 
which is exactly the inverse of the convexity bound for $L(\Pi\otimes\Omega,1/2)$. The phenomenon that the local integral for test vectors from (*) is of size $$\frac{1}{\text{convexity bound}}$$ is already observed in \cite{YH14} \cite{YH142} \cite{HMN16} for Rankin-Selberg integral and triple product formula, and applied to subconvexity and equidistribution problems. We see another example of the general phenomenon here. We will not use this to prove subconvexity in our setting, as it is already known for twisted L-function in general. 

We remark here that the author has given detailed descriptions for the local matrix coefficients of supercuspidal representations in \cite{YH14} and \cite{YH142}. These are however not enough to prove the nonvanishing results, which requires more knowledge on the special values of epsilon factors for different twists. So we shall use additional imput from compact induction theory to derive finer structures of supercupsidal representations. In particular we shall use the fact that when $p\neq 2$, the local supercuspidal representations are related to characters over a quadratic extension. As a result, the special value of epsilon factor can be written as a Gauss integral over quadratic extension (see Lemma \ref{thmoflocalepsilon}). So we have assumed that $p\neq 2$ for results in Proposition \ref{proptestvecinertsc}, \ref{propofexplicitI}, \ref{prop:testvecramifiedsc}.

Combining this paper with previous works and assuming $p\neq 2$, the only case where the test vector for Waldspurger's local integral is unknown is when $\pi_v$ is a minimal supercuspidal representation and $c(\pi_v)=c(\pi_{\Omega_v})$, where $\pi_{\Omega_v}$ is the representation of $\GL_2$ associated to $\Omega_v$.

To get (iii) from (ii), one just need to do a weighted average using $\Omega_v$. Actually if the epsilon value test doesn't fail, one can always find a test vector for nontrivial element in $\Hom_{\E_v^*}(\pi_v\otimes\Omega_v,\C)$  such that $\E_v^*$ acts on it by $\Omega_v^{-1}$. What is nontrivial is the additional invariance by a principal congruent subgroup (or its variant, see Definition \ref{defofcongsubgp}) inherited from the test vector in (ii) using Lemma \ref{lem:normalsubK}. In particular if we take $c(\Omega)=0$ and $\E_v$ to be inert, we will recover Gross and Prasad's test vector in our setting. The depth of the principal congruent subgroup is asymptotically $c(\pi_v)/2$, which we believe is best possible.

We note that the local integral for test vector from (iii) is obviously $1$ if the matrix coefficient and the Haar measure are properly normalized.

\subsection{Application to mass equidistribution of cusp forms on nonsplit torus}
The additional power saving in the local integral compared to the convexity bound was already noted in \cite{NPS14} and implicitly shown in \cite{YH142}. It is generalized and applied to prove power saving for the global Rankin-Selberg integral and triple product formula in \cite{HMN16} and used to prove subconvexity bound in a very general setting. 

In this paper we have observed similar behaviors, in particular the vanishing result in (i) for $\pi_v$ of level $c(\pi_v)\geq 2$ (supposing that $w_\pi=1$), and power saving result in Proposition \ref{prop:decaylocalint} when $c(\pi_v)\leq 1$. We shall make a quick application of these results to the mass equidistribution of cusp forms on nonsplit torus in depth aspect.

In particular let $f$ be an automorphic unitary cuspidal newform of finite conductor $N=q^c$ on $\GL_2$, with $L^2$ norm being $1$ and bounded archimedean components. Let $\E^*$ be a fixed nonsplit torus of $\GL_2$. We shall show in Theorem \ref{them:massequi} that the mass measure associated to $f$ is equidistributed on $\E^*$ as $c\rightarrow \infty$, in the sense that for any test function $\varphi$ on $\E^*$,
\begin{equation}
\int\limits_{[\E^*]}|f|^2(e)\varphi(e)de\rightarrow \int\limits_{[\E^*]}\varphi(e)de.
\end{equation}

While similar equidistribution-of-restriction (called quantum ergodic restriction problem by some literatures) results have been established in eigenvalue aspect in various settings as in \cite{DZ13} \cite{TZ13} \cite{MY13}, the author believe that this paper is the first to prove the mass equidistribution of restriction to nonsplit torus in level aspect.

We shall briefly sketch a proof here. To prove the equidistribution result, it will suffice to  test on characters $\Omega$ of $\E^*$. Using spectrum decomposition, we have
\begin{align}
&\int\limits_{[\E^*]}|f|^2(e)\Omega(e)de\\
=&\sum\limits_{\pi}\sum\limits_{\text{basis of cusp forms $\varphi\in\pi$}}<|f|^2,\varphi> \int\limits_{[\E^*]}\varphi(e)\Omega(e)de+<|f|^2,1> \int\limits_{[\E^*]}\Omega(e)de+\int\limits_{E}<|f|^2,E> \int\limits_{[\E^*]}E(e)\Omega(e)de.\notag
\end{align}
We want to show that the main term comes from the constant term $$<|f|^2,1> \int\limits_{[\E^*]}\Omega(e)de.$$ So we need to prove power saving for cuspidal spectrum and continuous spectrum. For simplicity we will just look at the cuspidal spectrum. The vanishing result in (i) allow us to reduce the sum in cuspidal representation $\pi$ to those such that $c(\pi)\ll c(\Omega)$, making the sum much shorter. Then the power saving either comes from the local integral for $\int\limits_{[\E^*]}\varphi(e)\Omega(e)de$ established in this paper, or the power saving in the local integral for $<|f|^2,\varphi>$ established in \cite{HMN16} (and partially noted in \cite{NPS14} and \cite{YH142}).

\subsection{Organization of the paper}
In Section 2 we will review basic notations and results. In particular we will quickly review the basic properties of compact induction theory for supercuspidal representations. 

In Section 3 we will prove vanishing and decaying results for Waldspurger's local integral when we pick improper test vectors from (*).

In Section 4 we will prove the existence of the test vectors from (*) for supercuspidal representations when the epsilon value test doesn't fail. We will also estimate the size of the local integral for such test vectors, and show that Gross-Prasad type test vector exists as a corollary.

In Section 5 we will go through similar process for induced representations. 
The choice of test vector is motivated by \cite{PN16}.

In Section 6 we give a quick application to the mass equidistribution on torus which make use of (i) above and also power saving result in \cite{HMN16}.

The author would like to thank Ameya Pitale for suggesting this problem, and Max-Planck Institute for Mathematics for support as most of the work is done during his visit there.

\section{Notations and previous results}
Let $\F$ be a number field and $\F_v$ be the corresponding local field of $\F$ at a place $v$. Let $O_v$ be the ring of integers of $\F_v$ and $\varpi_v$ be a local uniformizer. Let $q=|\varpi_v|_v^{-1}$.

Let $\E$ be a quadratic extension over $\F$. Suppose that $\E=\F(\sqrt{D})$ for an algebraic integer $D\in \F$. We fix an embedding $\E\hookrightarrow M_2(\F)$ as follows:
\begin{align}\label{quadraticembedding}
t=a+b\sqrt{D}&\mapsto \zxz{a}{b}{bD}{a}.
\end{align}
For simplicity we shall assume that $v(D)=0$ if $\E_v/\F_v$ is inert and $v(D)=1$ if $\E_v/\F_v$ is ramified. See Remark \ref{rem:vD} for what happens in general.

Let $\E_v$ to be the completion of $\E$ with respect to $v$. When $\E_v$ is a field extension over $\F_v$, let $O_{\E,v}$ be its ring of integers, and $\varpi_{\E,v}$ be a local uniformizer of $\E_v$. 
Let $e=e(\E_v/\F_v)$ to be the ramification index of $\E_v/\F_v$.

For an additive character $\psi_v$ over a local field $\F_v$, its level $c(\psi)$ is the least integer such that $\psi_v$ is trivial on $\varpi_v^{c(\psi)}O_v$.
We shall fix $\psi_v $ to be unramified (or level 0), and $\psi_{\E,v}=\psi_v\circ \Tr_{\E_v/\F_v}$. Then $\psi_{\E,v}$ as a function on $\E_v$ is of level $c_{\E_v}(\psi_{\E,v})=ec(\psi_v)-e+1$. In particular $c_{\E_v}(\psi_{\E,v})=0$ if $\E_v$ is an inert extension and $c_{\E_v}(\psi_{\E,v})=-1$ if $\E_v$ is a ramified extension.

For $\chi$ being a multiplicative character of $O_v^*$, its level $c(\chi)$ is the least integer such that $\chi$ is trivial on $1+\varpi_v^{c(\chi)}O_v$. When $\chi$ is trivial on $O_v^*$, we say it's unramified or level 0.
We denote by $\chi_\E$ the character of $O_{\E,v}^*$ defined by $\chi\circ N_{\E_v/\F_v}$. Similarly $c_{\E_v}(\chi_\E)=ec(\chi)-e+1$.

For all characters $\chi$ of $O_F^*$, we extend them to be characters on $\F^*$ by requiring that $\chi(\varpi)=1$ for the fixed uniformizer. For such characters, we denote $\int\limits_{v(x)=-j}\psi(x)\chi(x)d^*x$ simply by $$\int_{-j}\psi\chi.$$

When $2\nmid q$, we denote $\quadchar$ to be the unique nontrivial quadratic character of $O_v^*$.

Let $\lfloor a\rfloor$ be the largest integer which is smaller than $a$. Let $\lceil a\rceil$ be the least integer which is greater than $a$.

When $\E_v/\F_v$ is a field extension, we shall normalize the local Haar measure on $\F_v^*\backslash\E_v^*$ such that $\Vol(\F_v^*\backslash\E_v^*)=1$

\subsection{Theorem of Tunnell-Saito and Waldspurger's formula}\label{Sec:Waldspurger}
Let $\pi$ be an automorphic cuspidal representation of $\GL_2$ over $\F$. 
Let $\B$ be a division algebra over $\F$ and let $\sigma $ be an automorphic representation of $\B^*$ whose Jacquet-Langlands image is $\pi$. Note that we allow $\B=M_2(\F)$ and $\sigma=\pi$. Let $F\in \sigma$. We shall also fix a embedding $\E\hookrightarrow \B$ in general, not to be specified here.
Let $\Omega$ be a Hecke character over $\E$ such that $\Omega|_{\A^*_\F}\cdot w_{\pi}=1$. We consider the following global integral
\begin{equation}
\int\limits_{Z_\A\E^*\backslash \A_\E^*}F(e)\Omega(e)de.
\end{equation}

This period integral actually gives an element in $\Hom_{\A_\E^*}(\sigma\otimes\Omega,\C)$. But it's not necessary that this space is non-zero. The local obstruction for the global period integral to be nonzero is described by an epsilon value test.

Let the Hasse invariant $\epsilon(\B_v)$ of a local quaternion algebra $\B_v$ be $1$ if $\B_v\simeq M_2(\F_v)$, and $-1$ if it's a division algebra. Let $\Pi$ be the base change of $\pi$ to $\E$.
The following theorem is due to Tunnell and Saito (\cite{Sa} \cite{Tu83}).
\begin{theo}\label{Tunnell}
The space $\Hom_{\E_v^*}(\sigma_v\otimes \Omega_v,\C)$ is at most one-dimensional. It is nonzero if and only if 
\begin{equation}
\epsilon(\frac{1}{2},\Pi_{v}\otimes\Omega_v)=\Omega_v(-1)\epsilon(\B_v).
\end{equation}
\end{theo}

When there is no local obstruction, Waldspurger proved the following result in \cite{Walds}

\begin{theo}\label{WaldsMCversion}
Let $F_1\in \sigma$ and $F_2\in \hat{\sigma}$. Then
\begin{equation}
\frac{\int\limits_{Z_{\A}\E^{*}\backslash \A_\E^*} F_{1}(e) \Omega(e)de\int\limits_{Z_{\A}\E^{*}\backslash \A_\E^*} F_{2}(e)\Omega^{-1}(e) de}{(F_1,F_2)}=\frac{\zeta(2)L(\Pi\otimes\Omega,1/2)}{2L(\pi,Ad,1)}\prod\limits_{v}P^0_v.
\end{equation}
where 
\begin{equation}\label{formula2.4.3}
P^0_v=\frac{L_v(\pi,Ad,1)L_v(\eta,1)}{\zeta_v(2)L_v(\Pi\otimes\Omega^{-1},1/2)}
  \frac{\int\limits_{\F_v^*\backslash \E_v^*}<\sigma_v(e)F_{1,v},F_{2,v}>\Omega_v(e)de}{<F_{1,v},F_{2,v}>}.
\end{equation}
\end{theo}

We will mainly be interested in the case when $\B=M_2(\F)$.

In the following we shall denote $I(\Phi_v,\Omega_v)$ to be the local integral in Waldspurger's formula:
\begin{equation}
I(\Phi_v,\Omega_v)=\int\limits_{\F_v^*\backslash\E_v^*}\Phi_v(e)\Omega_v(e)de=0,
\end{equation}
where $\Phi_v$ is the matrix coefficient associated to elements of $\pi_v$ (to be specified later on) and always normalized such that $\Phi_v(1)=1$.

\subsection{Local integrals}
From this section on we will mainly work locally so we shall omit most of sub-index $v$ without confusion.
\begin{lem} \label{Iwasawadecomp}
For every positive integer $c$,
$$\GL_2(\F )=\coprod\limits_{0\leq i\leq c} B\zxz{1}{0}{\varpi ^i}{1}K_0(\varpi ^c).$$
Here $B$ is the Borel subgroup of $\GL_2$. 
\end{lem}

\begin{defn}\label{defofcongsubgp}
For $k\geq 1$, let $K_0(\varpi^k)$ denote following the compact subgroup of $\GL_2(O_v)$
$$\{g\in \zxz{O_v^*}{O_v}{\varpi^k O_v}{ O_v^*} \}.$$ 
Let $K_1^1(\varpi ^k,\varpi ^k)$ denote the compact subgroup
$$\{g\in \zxz{1+\varpi^k O_v}{\varpi^k O_v}{\varpi^k O_v}{1+\varpi^k O_v} \}.$$ 
Let $K_1^1(\varpi ^{k+1},\varpi ^k)$ denote the compact subgroup
$$\{g\in \zxz{1+\varpi^k O_v}{\varpi^k O_v}{\varpi^{k+1} O_v}{1+\varpi^k O_v} \}.$$
\end{defn}

\begin{lem}\label{lem:normalsubK}
If $t\in \E^*$ is an element from inert quadratic extension where $v(D)=0$, then
$$ t^{-1}K_1^1(\varpi ^k,\varpi ^k)t=K_1^1(\varpi ^k,\varpi ^k).$$
Similarly if $t\in \E^*$ is an element from ramified quadratic extension where $v(D)=1$, then
$$ t^{-1}K_1^1(\varpi ^{k+1},\varpi ^k)t=K_1^1(\varpi  ^{k+1},\varpi ^k).$$

\end{lem}

\begin{lem}\label{lemofGaussint}
Let $m\in\F $ such that $v(m)=-j<0$, and $\mu$ be a character of $O_{v}^*$ of level $k>0$. Then
\begin{equation}
  |\int\limits_{v(x)=0}\psi (mx)\mu^{-1}(x)d^*x|=\begin{cases}
                                                   \sqrt{\frac{q}{(q-1)^2q^{k-1}}},&\text{\ if\ }j=k;\\
                                                   0,&\text{\ otherwise.}
                                                  \end{cases}
\end{equation}
\end{lem}

\begin{lem}
Let $\chi$ be a multiplicative character of $O_v^*$ of level $j$ and $v(b)\geq 0$. Suppose that $p\neq 2$. Then $\chi(\frac{b^2D}{1-b^2D})$ is still of level $j$ in $b$, unless $\chi=\quadchar$ is the unique character of order 2; 
\end{lem}

\begin{defn}\label{lemofstructureofchar}
Let $\chi$ be a character of $O_v^*$ such that $c(\chi)\geq 2$. Then there exists a unit $\alpha_\chi$ associated to $\chi$ such that
\begin{equation}
\chi(1+x)=\psi (\frac{\alpha_\chi}{\varpi ^{c(\chi)}}x)
\end{equation}
for any $x\in \varpi ^{\lceil \frac{c(\chi)}{2}\rceil}O_v$.
\end{defn}

\begin{lem}\label{lem:surjective}
$\alpha_\chi$ is unique up to $\varpi ^{\lfloor \frac{c(\chi)}{2} \rfloor}$. The map 
\begin{align}
\{\text{level $k$ characters}\}&\rightarrow  O_v^*/1+\varpi ^{\lfloor \frac{k}{2} \rfloor}O_v\\
\chi&\mapsto \alpha_\chi\notag
\end{align}
is surjective.
\end{lem}
\begin{proof} $\alpha_{\chi_1}\equiv \alpha_{\chi_2}$ if and only if $\chi_1\chi_2^{-1}$ of level $\leq\lceil \frac{c(\chi)}{2}\rceil$. Then one just need to do a simple counting.
\end{proof}

\begin{lem}\label{lemoftwocharGaussint}
Let $\chi$, $\nu$ be two multiplicative characters of $\F^*$ extended from $O_v^*$, such that $c(\chi)\geq 2c(\nu)$. Let $\alpha_\chi$ be the constant associated to $\chi$ as above. Then
\begin{equation}
\int_{-c(\chi)}\chi\nu\psi=\nu(-\frac{\alpha_\chi}{\varpi^{c(\chi)}})\int_{-c(\chi)}\chi\psi.
\end{equation}
\end{lem}

\begin{proof}
Use stationary phase analysis.
\end{proof}

\begin{lem}\label{lemofchiE}
Let $\E /\F $ be a quadratic field extension with  ramification index $e$.
Let $\chi$ be a character of $O_v^*$ of level $c(\chi)\geq 2$ associated to a constant $\alpha_\chi$.  Recall that $\chi_\E$ is of level $ec(\chi)-e+1$, $\psi_{\E }$ is of level $ec(\psi)-e+1=1-e$. Then
\begin{equation}
\chi_E(1+x)=\psi_{\E }(\frac{\alpha_\chi}{\varpi ^{c(\chi)}}x)
\end{equation}
for all $x\in \varpi_{\E }^{\lceil\frac{ec(\chi)-e+2}{2}\rceil}O_{\E }$
%\begin{equation}
%\alpha_\chi\equiv \alpha_\chi'\mod \varpi_\E^{\lfloor\frac{ec(\chi)-e+1}{2}\rfloor}
%\end{equation}
\end{lem}
\begin{proof}
%We only prove the last statement here.
For $x\in \varpi_{\E }^{\lceil\frac{ec(\chi)-e+2}{2}\rceil}O_{\E }$, we have
\begin{equation}
\chi_\E(1+x)=\chi(1+\Tr_{\E/\F}(x))=\psi(\frac{\alpha_\chi}{\varpi ^{c(\chi)}}\Tr_{\E /\F }(t))=\psi_\E(\frac{\alpha_\chi}{\varpi^{c(\chi)}}x).
\end{equation}
%On the other hand by definition and that $\varpi_\E^{e}=\varpi$,
%\begin{equation}
%\chi_\E(1+t)=\psi_\E(\frac{\alpha_\chi'}{\varpi^{c(\chi)}}t).
%\end{equation}
%Then the conclusion follows by requiring there two expressions to be equal for all $t\in \varpi_\E^{\lceil\frac{ec(\chi)-e+2}{2}\rceil}O_\E$.
\end{proof}

\subsection{supercuspidal representations}
We will work purely locally for this section so we shall remove $v$ from sub-index without confusion.
\subsubsection{Kirillov model}
 Let $\pi$ be a supercuspidal representation over $\F$, with central character $w_\pi$.
Its Kirillov model can be realized on $S(\F^*)$ such that
\begin{equation}
\pi(\zxz{a_1}{m}{0}{a_2})\varphi(x)=w_{\pi}(a_2)\psi(ma_2^{-1}x)\varphi(a_1a_2^{-1}x),
\end{equation}

 A basis of this representation can be given by
\begin{equation}\label{formulabasissc}
\charf_{\nu,n}(x)=\begin{cases}
                        \nu(u), &\text{if\ } x=u\varpi^n\text{\  for\ } u\in O_F^*;\\
						0,&\text{otherwise}.
                       \end{cases}
\end{equation}
The action of $\omega=\zxz{0}{1}{-1}{0}$ on this basis is given by
\begin{equation}\label{eq:supercuspidalomegaaction}
\pi(\omega )\charf_{\nu,n}=C_{\nu w_0^{-1}}z_0^{-n}\charf_{\nu^{-1}w_0,-n+n_{\nu w_0^{-1}}},
\end{equation}
where  $z_0=w_{\pi}(\varpi)$ and $w_0=w_{\pi}|_{O_F^*}$. Recall that $n_\nu\leq -2$ and $c=-n_1$.

The relation $\omega^2=-\zxz{1}{0}{0}{1}$ implies that
\begin{equation}\label{doubleactionresult}
n_{\nu^{-1}}=n_{\nu w_0^{-1}},\text{\ \ } C_{\nu^{-1}} C_{\nu w_0^{-1}}=w_0(-1)z_0^{n_{\nu^{-1}}}.
\end{equation}

The newform in supercuspidal representation is simply $\charf_{1,0}$.

The numbers $C_\nu$ and $n_\nu$ can be related to epsilon factor and level of twisted representations. For simplicity we fix a uniformizer $\varpi$ and extend $\nu$ be a character of $F^*$ by requiring $\nu(\varpi)=1$.
The action above can be equivalently formulated  as
 $$\pi(\omega)\charf_{\nu,n}=\epsilon(\pi\otimes\nu^{-1},\psi,1/2)z_0^{-c(\pi\otimes\nu^{-1})-n}\charf_{w_0\nu^{-1},-c(\pi\otimes\nu^{-1})-n},$$
that is,
\begin{equation}\label{Cs&epsilon}
 C_{\nu w_0^{-1}}=\epsilon(\pi\otimes\nu^{-1},\psi,1/2)z_0^{-c(\pi\otimes\nu^{-1})},
 n_{\nu w_0^{-1}}=-c(\pi\otimes\nu^{-1}).
\end{equation}

The constant $n_\nu$ is easier to describe. In particular when $w_\pi=1$ and $\nu$ is of level $i$, $n_\nu=-\max\{c(\pi),2i\}$. (See for example \cite{YH14}.) 

Let
\begin{equation}
\Phi(g)=<\pi(g)\varphi,\varphi>
\end{equation}
be the matrix coefficient associated to the newform $\varphi$.
It is right $K_0(\varpi^{c})-$invariant. 
By Lemma \ref{Iwasawadecomp}, to understand $\Phi(g)$, it will be enough to understand $\Phi(\zxz{a}{m}{0}{1}\zxz{1}{0}{\varpi^i}{1})$ for $0\leq i\leq c$. For the following we shall denote $$\Phi^{(i)}(a,m)=\Phi(\zxz{a}{m}{0}{1}\zxz{1}{0}{\varpi^i}{1}).$$
\begin{rem}
Note that for fixed valuation for $a$ and $m$, $\Phi^{(i)}(a,m)$ only depends on $\frac{m}{a}$, as $\Phi$ is actually bi-$K_0(\varpi^c)$ invariant. So we can think of it as a one-parameter function and talk about its levels.
\end{rem}

\begin{prop}\label{propofsupportofMC}
Let $\Phi$ be the matrix coefficient associated to the newform of a minimal supercuspidal representation. 
\begin{enumerate}
\item[(i)] For $c-1\leq i\leq c$, $\Phi^{(i)}(a,m)$ is supported on $v(a)=0$ and $v(m)\geq -1$. On the support, we have
\begin{equation}
\Phi^{(i)}(a,m)=\begin{cases}
1,&\text{\ if\ }v(m)\geq 0 \text{\ and\ }i=c;\\
-\frac{1}{q-1},&\text{\ if\ }v(m)=-1 \text{\ and\ }i=c;\\
-\frac{1}{q-1},&\text{\ if\ }v(m)\geq 0 \text{\ and\ }i=c-1.\\
\end{cases}
\end{equation}
When $v(a)=0$, $v(m)=-1$ and $i=c-1>1$, 
$\Phi^{(c-1)}(a,m)$ consists of level 1 and  level 0 components.
\item[(ii)] For $0\leq i<c-1$, $\Phi^{(i)}(a,m)$ is supported on $v(a)=\min\{0,2i-c\}$, $v(m)=i-c$, consisting of level $c-i$ components.
\end{enumerate}
\end{prop}
\begin{rem}\label{rem:supportofMC}
See \cite{YH14} for more general setting. When $\pi$ is an induced representation from two ramified characters of same level or non-minimal supercuspidal representation, we have essentially same conclusion, except that $\Phi^{(i)}(a,m)$ could be supported on $v(a)\geq 0$ when $i=c(\pi)/2$.

\end{rem}
This information is however not enough for the purpose of this paper, and we need input from compact induction theory to tell us about $C_\nu$.

\subsubsection{Compact induction theory and epsilon factors}

It was shown in \cite{BH14} that all supercuspidal representations can be constructed as an induced representation from a representation of a compact subgroup of $\GL_2$. The results here are directly taken from \cite{BH14}, though we use a different convention for the levels from that of \cite{BH14} and readers should be aware of this.

In short, when $p\neq 2$, supercuspidal representations are related to character $\theta$ defined over a quadratic field extension. When $\pi=\pi_\theta$ is of level $\varpi^{2k}$, $\theta$ is level $\varpi_\E^k$ over an inert quadratic extension $\E$; When $\pi_\theta$ is of level $\varpi^{2k+1}$, $\theta$ is of level $\varpi_E^{2k}$ over a ramified quadratic extension $\E$.

Further this association satisfies:
\begin{enumerate}
\item $w_{\pi_\theta}=\theta|_{F^*}$.
\item $\pi_{\hat{\theta}}=\hat{\pi}_\theta$.
\item if $\eta $ is a character of $F^*$, then $\pi_{\chi\eta_E}=\eta\pi_\chi$, where $\eta_E=\eta\circ N_{E/F}$.
\end{enumerate}

%\begin{lem}
%The level $c(\pi)$ of a supercuspidal representation is related to its depth by the following simple relation:
%\begin{equation}
%l(\pi)=\frac{c(\pi)}{2}-1.
%\end{equation}
%\end{lem}

The supercuspidal representation $\pi$ is called minimal if its level is minimal among twists. In particular the supercuspidal representations with trivial central characters are minimal.
When $\pi_\chi$ is minimal, there is a simple relation between  the datum. If $c(\pi)$ is odd, then $E/F$ is ramified and $\chi$ is of level $c(\pi)-1$; If $c(\pi)$ is even, then $E/F$ is unramified and $\chi$ is of level $c(\pi)/2$.

%For the depth of a supercuspidal representation $l(\pi)$ being even or odd, the meaning of the following notations and the datum for the constructions are twofold.

%When $l(\pi)$ is even, $\Pi=\zxz{\varpi}{0}{0}{\varpi}$, $\mathcal{U}=\zxz{ \GL_2(O^*)$, $U_m=1+\Pi^n$

%When $p\neq 2$, the supercuspidal representations can be parametrized by so-called admissible pairs.

Now we introduce the basic lemma on the epsilon factors for supercuspidal representations when $p\neq 3$.

\begin{lem}\label{thmoflocalepsilon}
Let  $\theta$ be of level $n$ over a quadratic local field extension $\E/\F$ with ramification index $e$. Let $\pi=\pi_\theta$ be the associated supercuspidal representation. Suppose that $\pi$ is minimal. 
Denote $q_\E$ to be the order of residue field for $\E$ so that $q_\E=q^{2/e}$. 
Let $\eta$ be a character of $F^*$ with level $m$ such that 
%\begin{enumerate}
%\item[(1)]Suppose that 
the level of $\eta_E$ satisfies $em-e+1\leq c(\theta)$, then 
\begin{equation}
\epsilon(\pi\otimes\eta,1/2,\psi)=(-1)^{e(E/F)n}\frac{q_\E-1}{q_\E}\sqrt{q_\E^n}\int\limits_{E^*}(\theta\eta_E)^{-1}(x)\psi_E(x) d^*x.
\end{equation}
%\item[(2)]If $em-e+1> c(\theta)$, then
%\begin{enumerate}
%\item[i)]If $E/F$ is ramified, suppose that $\eta(1+x)=\psi(cx)$ for $x\in \varpi^{[\frac{m+1}{2}]}$ and $c\in F^*$ with $v(c)=-m+1$. 
%\begin{equation}
%\epsilon(\pi\otimes\eta,1/2,\psi)=\epsilon_{E/F}(c)\lambda_{E/F}(\psi) \frac{q_\E-1}{q_\E}\sqrt{q_\E^{2m-1}}\int\limits_{E^*}(\theta\eta_E)^{-1}(x)\psi_E(x) d^*x
%\end{equation}
%\item[ii)]If $E/F$ is unramfied, 
%\begin{equation}
%\epsilon(\pi\otimes\eta,1/2,\psi)=(-1)^{e(E/F)m}\frac{q_\E-1}{q_\E}\sqrt{q_\E^m}\int\limits_{E^*}(\theta\eta_E)^{-1}(x)\psi_E(x) d^*x.
%\end{equation}
%\end{enumerate}
%\end{enumerate}

%Then 
%\begin{equation}\label{Mainlocalepsilonlevelone}
%\epsilon(\pi,1/2,\psi)=(-1)^{e(E/F)n}\frac{q_\E-1}{q_\E}\sqrt{q_\E^n}\int\limits_{E^*}\chi^{-1}(x)\psi_E(x) d^*x.
%\end{equation}
\end{lem}

\begin{cor}\label{corofquotientofC}
Suppose that $\pi=\pi_\theta$ is a minimal supercuspidal representation associated to a character $\theta$ defined over a quadratic field extension $\E$ with ramification index $e$. Let $\alpha_\theta$ be the constant associated to $\theta$ as in Definition \ref{lemofstructureofchar}. Let $\eta $ and $\nu$ be multiplicative characters over $\F$. Assume that $ec(\nu)-e+1\leq c(\theta)/2$, $ec(\eta)-e+1\leq c(\theta)$, and let $\alpha_\eta$  be the constant associated to $\eta$. Then
\begin{equation}
\frac{C_{\nu\eta^{-1}}}{C_{\eta^{-1}}}=\nu_\E(\frac{\alpha_\theta+\alpha_\eta \varpi_\E^{c(\theta)+e-1}\varpi^{-c(\eta)}}{\varpi_\E^{c(\theta)+e-1}})
\end{equation}
\end{cor}

\begin{proof}
By previous lemma,
\begin{align}
\frac{C_{\nu\eta^{-1}}}{C_{\eta^{-1}}}=\frac{\epsilon(\pi_\theta\otimes\nu^{-1}\eta,1/2,\psi)}{\epsilon(\pi_\theta\otimes\eta,1/2,\psi)}=\frac{\int\limits_{\E^*} \theta^{-1}(\nu\eta^{-1})_{\E}\psi_E}{\int\limits_{\E^*} \theta^{-1}\eta^{-1}_{\E}\psi_E}
\end{align}

Note that by definition and  Lemma \ref{lemofchiE}
\begin{equation}
\theta^{-1}\eta^{-1}_\E(1+x)=\psi_\E(-\frac{\alpha_\theta}{\varpi_\E^{c(\theta)+e-1}}x)\psi_\E(-\frac{\alpha_\eta}{\varpi^{c(\eta)}}x)=\psi_\E(\frac{-\alpha_\theta-\alpha_\eta \varpi_\E^{c(\theta)+e-1}\varpi^{-c(\eta)}}{\varpi_\E^{c(\theta)+e-1}}x)
\end{equation}
for $x\in \varpi_\E^{\lceil \frac{c(\theta)}{2}\rceil}O_\E$.  Note that since $ec(\eta)-e+1\leq c(\theta)$ and $\pi_\theta$ is minimal
$$-\alpha_\theta-\alpha_\eta \varpi_\E^{c(\theta)+e-1}\varpi^{-c(\eta)}\in O_\E^*.$$

Then we apply Lemma \ref{lemoftwocharGaussint} to get
\begin{align}
\frac{\int\limits_{\E^*} \theta^{-1}(\nu\eta^{-1})_{\E}\psi_E}{\int\limits_{\E^*} \theta^{-1}\eta^{-1}_{\E}\psi_E}=\frac{\int\limits_{-(c(\theta)+e-1)} \theta^{-1}(\nu\eta^{-1})_{\E}\psi_E}{\int\limits_{-(c(\theta)+e-1)} \theta^{-1}\eta^{-1}_{\E}\psi_E}=\nu_\E(\frac{\alpha_\theta+\alpha_\eta \varpi_\E^{c(\theta)+e-1}\varpi^{-c(\eta)}}{\varpi_\E^{c(\theta)+e-1}}).
\end{align}
\end{proof}

\subsection{Twisted representation, local integral and matrix coefficient for twisted elements}\label{secoftwist}
Let $(\pi,V)$ be a local representation of $\GL_2$ realized in the linear space $V$. For a multiplicative character $\chi$, suppose that $\chi(\varpi)=1$. We shall write $\chi(g)$ to mean $\chi(\det g)$.
The representation $\pi'=\pi\otimes\chi$ can be realized in the same space, with the action
\begin{equation}\label{twistedaction}
\pi'(g)v=\chi(g)\pi(g)v.
\end{equation}

However we need to be more careful if we want to keep working with, for example, induced model or Whittaker model. Let $\iota_\chi$ denote the following map
\begin{align}
\Ind(\chi_1,\chi_2)&\rightarrow \pi'=\Ind(\chi_1\chi,\chi_2\chi)\\
f&\mapsto f(g)\chi(g).\notag
\end{align}
Note that this is not with respect to group actions. We endow the image of $\iota_\chi$ with the action of $\pi$ by forcefully require $\iota_\chi$ to be group homomorphism. Then one can easily see that
\begin{equation}
\pi'(g_0)\iota_\chi(f)=f(gg_0)\chi(gg_0)=\pi(g_0)\iota_\chi(f)\chi(g_0).
\end{equation}
So the relation between $\pi'$ and $\iota_\chi(\pi)$ agrees with the identification in \ref{twistedaction}. 

Similarly for elements in Whittaker models or Kirillov models, we define the map $\iota_\chi$ as
\begin{align}
W&\mapsto W(g)\chi(g)\\
\varphi&\mapsto \varphi(x)\chi(x).\notag
\end{align}
\begin{defn}\label{def:twistednewform}
In this paper we will care about elements in $\pi$ which are images under $\iota_\chi$ of newform in $\pi\otimes\chi^{-1}$. We call such elements twisted newforms.
\end{defn}
Now we show how this twisting can be used for computing matrix coefficients and p-adic integrals.

Suppose that $\pi$ is unitary, and $(\cdot,\cdot)_\pi$ is an invariant unitary pairing for $\pi$. For elements from Whittaker model, the pairing is 
\begin{equation}
(W_1,W_2)_\pi=\int\limits_{\F^*}W_1(\zxz{x}{0}{0}{1})\overline{W_2}(\zxz{x}{0}{0}{1})d^*x.
\end{equation}
Then one can easily check that
\begin{equation}
(\iota_\chi (W_1),\iota_\chi (W_2))_{\pi'}=(W_1,W_2)_\pi.
\end{equation}
As a consequence, if we let $\Phi(g)$ be the matrix coefficient associated to $W_1,W_2$, and let $\Phi'$ be the matrix coefficient associated to $\iota_\chi(W_i)$, then
\begin{equation}\label{formulatwistMC}
\Phi'(g)=\Phi(g)\chi(g).
\end{equation}

%\todo{reduce to minimal representations for $\pi$}
Now with the same notations,
\begin{equation}\label{eq:twistedlocalint}
I(\Phi,\Omega)=\int\limits_{\F^*\backslash\E^*}\Phi(e)\Omega(e)de=\int\limits_{\F^*\backslash\E^*}\Phi(e)\chi(\det(e))\Omega\chi_\E^{-1}(e)de=I(\Phi',\Omega\chi_\E^{-1}).
\end{equation}
This means that instead of looking for test vectors for the pair $(\pi,\Omega)$, we can solve the same problem for the pair $(\pi\otimes \chi, \Omega\chi_\E^{-1})$. In particular we can assume that $\pi$ is a minimal representation. 
\begin{rem}\label{rem:twistforlocalint}
\added{Further for $\pi$ being a supercuspidal representation and $p\neq 2$, we can assume that its central character is either unramified or level $1$.  Let $\alpha$ be the constant associated to the central character of $\pi$. Then by Lemma \ref{lem:surjective}, there exists a character $\chi$ whose associated constant is $-\frac{\alpha}{2}$. As a result, the central character for $\pi\otimes\chi$ will have smaller level. Repeat this process for finite steps and one will get a central character of level $\leq 1$.}
\end{rem}

\subsection{Iwasawa decomposition for conjugated torus}
Note that $\charf_{\eta,-d}=\pi(\zxz{\varpi^{d}}{0}{0}{1})\charf_{\eta,0}$. So if we know the matrix coefficient for $\charf_{\eta,0}$, we can just do a conjugation to get the matrix coefficient for $\charf_{\eta,-d}$. For the local integral $I(\Phi,\Omega)$, this conjugation of matrix coefficient is equivalent to a conjugated embedding of $\E^*$ into $\GL_2$ via a change of variable. This motivates the following consideration.

Suppose that $\E=\F(\sqrt{D})$ is embedded into the matrix algebra via
$$ a+b\sqrt{D}\mapsto \zxz{a}{b}{bD}{a}.$$
For $g=\zxz{\varpi^d}{0}{0}{1}$,
$$g^{-1}eg=\zxz{a}{b\varpi^{-d}}{bD\varpi^{d}}{a}.$$
For later use we study here the Iwasawa decomposition for this matrix.
\begin{lem}\label{lemofcosetdecomp}
If  $v(\frac{bD\varpi^d}{a})\geq 0$, then let $i= v(\frac{bD\varpi^d}{a})$ and
\begin{equation}\label{eq:cosetdecomp}
\zxz{a}{b\varpi^{-d}}{bD\varpi^{d}}{a}=a\zxz{\frac{a^2-b^2D}{abD}\varpi^{i-d}}{\frac{b}{a\varpi^d}}{0}{1}\zxz{1}{0}{\varpi^i}{1}\zxz{\frac{bD}{a\varpi^{i-d}}}{0}{0}{1}.
\end{equation}
If $v(\frac{bD\varpi^d}{a})< 0$, then $i=0$ and
\begin{equation}
\zxz{a}{b\varpi^{-d}}{bD\varpi^{d}}{a}=bD\varpi^d\zxz{\frac{a^2-b^2D}{b^2D^2\varpi^{2d}}}{\frac{a}{bD\varpi^d}-\frac{a^2-b^2D}{b^2D^2\varpi^{2d}}}{0}{1}
\zxz{1}{0}{1}{1}\zxz{1}{-1+\frac{a}{bD\varpi^d}}{0}{1}.
\end{equation}

\end{lem}
\begin{rem}\label{rem:vD}
We assume from now on that $v(D)=0 $ if $\E$ is an inert extension and $v(D)=1$ if $\E$ is a ramified extension. Otherwise one can do a conjugation to change the embedding. So if one change test vectors correspondingly, one will get exactly the same calculation for local integral.
\end{rem}
%\begin{prop} %vanishing results
%Let $\pi$ be a supercuspidal representation or induced representation with trivial central character of level $2k$ or $2k+1$. 
%Let $\Omega $ be a character defined over a quadratic field extension such that $c(\pi_\Omega)<c(\pi)$.
%Let $\Phi$ be the matrix coefficient associated to  a test vector $\charf_{\eta,d}$. If $d\neq k$, then $$I(\Phi,\Omega)=0.$$
%\end{prop}

\section{Vanishing and decaying results for test vectors}
In this section we will show that the local integral of Waldspurger's period integral will either be vanishing  or quickly decay in size if we don't pick proper test vectors.

We first prove a vanishing result for most of test vectors when $\pi$ is of large level.
\begin{prop}\label{prop:vanishinglocalint}
Let $\pi$ be a supercuspidal representation or induced representation defined by two ramified characters of same level, and $c(\pi)=2k$ or $2k+1$. Let $\Omega $ be a character of $\E^*$ such that $\frac{2}{e}c(\Omega)< c(\pi)$. Let $\varphi^0$ be the newform for $\pi$ and $\Phi$ be the matrix coefficient associated to $\pi(\zxz{\varpi^{d}}{0}{0}{1})\varphi^0 $. Suppose that $v(D)=0$ if $\E$ is an unramified extension and $v(D)=1$ if it's ramified.
Then if $d\neq k$,
\begin{equation}
I(\Phi,\Omega)=\int\limits_{\F^*\backslash\E^*}\Phi(e)\Omega(e)de=0.
\end{equation}

\end{prop}

\begin{proof}
Let $\Phi^0$ be the matrix coefficient associated to the newform of $\pi$, which is described in Proposition \ref{propofsupportofMC} and Remark \ref{rem:supportofMC}. Then
\begin{equation}
I(\Phi,\Omega)=\int\limits_{\F^*\backslash\E^*}\Phi^0(g^{-1}eg)\Omega(e)de=0
\end{equation}
for $g=\zxz{\varpi^{d}}{0}{0}{1}$. We shall compare the support of $\Phi^{0,(i)}(a,m)$ as in Proposition \ref{propofsupportofMC} with the Iwasawa decomposition of conjugated torus given in Lemma \ref{lemofcosetdecomp}. 

The first case is when $d>k$. If $v(\frac{bD\varpi^d}{a})\geq 0$, $$v(m)=v(\frac{b}{a\varpi^d})=i-v(D)-2d<i-c(\pi).$$
If $v(\frac{bD\varpi^d}{a})<0$,
$$v(m)=v(\frac{a}{bD\varpi^d}-\frac{a^2-b^2D}{b^2D^2\varpi^{2d}}   )=-v(D)-2d<-c(\pi).$$
So the conjugated torus completely misses the support of matrix coefficient.

Consider the case $d<k$ now. Then we will get reversed inequalities like above. So the conjugated torus will miss the support except when $i=c(\pi)$ and $c(\pi)-1$, where we know the value of $\Phi^0$ very explicitly according to Proposition \ref{propofsupportofMC}. Thus
\begin{equation}
I(\Phi,\Omega)=\int\limits_{\F^*\backslash\{v(\frac{bD\varpi^d}{a})\geq c(\pi)\}}w_\pi(a)\Omega(a+b\sqrt{D})de+\int\limits_{\F^*\backslash\{v(\frac{bD\varpi^d}{a})=c(\pi)-1\}}(-\frac{1}{q-1})w_\pi(a)\Omega(a+b\sqrt{D})de.
\end{equation}
By the assumption on $c(\Omega)$, $w_\pi(a)\Omega(a+b\sqrt{D})=1$ on above domains.
One can also check (including split extension) that 
\begin{equation}
\Vol(\F^*\backslash\{v(\frac{bD\varpi^d}{a})=c(\pi)-1\}  )=(q-1)\Vol(\F^*\backslash\{v(\frac{bD\varpi^d}{a})\geq c(\pi)\} ).
\end{equation} 
So $I(\Phi,\Omega)=0$ as long as $d\neq k$.
\end{proof}
\begin{rem}\label{rem:vanishinglocalint}
The argument for the case $d<k$ also works for the case $\frac{2}{e}c(\Omega)\geq c(\pi)$ if $d=-n$ for $n\geq c(\Omega)$. If one choose $d$ properly such that $\Omega$ is trivial on $\{v(\frac{bD\varpi^d}{a})\geq c(\pi)\}$ but not trivial on $\{v(\frac{bD\varpi^d}{a})=c(\pi)-1\} $, one will get test vector for $\Hom_{\E^*}(\pi\otimes\Omega,\C)$ when $c(\Omega)$ is sufficiently large.

Using (\ref{eq:twistedlocalint}), one can extend the vanishing results to more general selections of test vectors.
\end{rem}
\begin{rem}
This result also gives evidence why we will not find test vector if, for example, $\E$ is inert extension and $c(\pi)=2k+1>2c(\Omega)$. This is because we will never be able to take $d$ such that
\begin{equation}
v(m)=i-v(D)-2d=i-c(\pi)
\end{equation}
\end{rem}

Next we prove a power saving result as we vary test vectors when $\pi$ is of smaller level.

\begin{prop}\label{prop:decaylocalint}
Let $\pi$ be an unramified or special unramified representation. Let $\Omega $ be a fixed character of $\E^*$. Let $\varphi^0$ be the newform for $\pi$ and $\Phi$ be the matrix coefficient associated to $\pi(\zxz{\varpi^{d}}{0}{0}{1})\varphi^0 $. %Suppose that $v(D)=0$ if $\E$ is an unramified extension and $v(D)=1$ if it's ramified. 
Then there exists a positive number $\delta>0$ such that for $d=-n$,
\begin{equation}
I(\Phi,\Omega)\ll q^{-\delta n}
\end{equation}
as $n\rightarrow \infty.$
\end{prop}
\begin{proof}
It's actually possible to figure out $\delta$ explicitly. For our application however, we will just show that such $\delta$ exists. This result essentially follows from the decay of matrix coefficient in general. For simplicity we only work with the case of unramified representation, whereas the case of special unramified representation case is similar and a little more complicated.

Recall that the matrix coefficient $\Phi^0$ for a spherical element is bi-$K$ invariant and we have the following decay of matrix coefficient
\begin{equation}
  \Phi^0(\zxz{x}{0}{0}{1})\ll q^{(\alpha-1/2+\epsilon)|v(x)|}
 \end{equation}
Here $\alpha$ is a bound towards Ramanujan conjecture and we can take $\alpha=7/64$. This formula is actually true for fixed general element in general representation. We shall fix a small $\epsilon$ and take $\delta_0=-\alpha+1/2-\epsilon>3/8$ so that 
\begin{equation}\label{eq:decayofMC}
 \Phi^0(\zxz{x}{0}{0}{1})\ll q^{-\delta_0|v(x)|}.
\end{equation}

Assume that $\E$ is a field extension first. Let $\Phi$ be the matrix coefficient associated to the test vector $\pi(g)\varphi^0$ for $g= \zxz{\varpi^{-n}}{0}{0}{1}$. %is invariant under $1+\varpi_\E^nO_\E$. 
We can assume, with proper twisting like (\ref{eq:twistedlocalint}), that $\pi$ and $\Omega$ are unitary. Then
\begin{align}
|I(\Phi,\Omega)|\leq &\int\limits_{\F^*\backslash\E^*} |\Phi(e)|de%=\int\limits_{O_\F^*\backslash O_\E^*} |\Phi(e)|de
\\
=&\int\limits_{\F^*\backslash\{v(b)-v(a)\geq n/2\}} |\Phi^0 (g^{-1}(a+b\sqrt{D})g)| de
+\int\limits_{\F^*\backslash\{v(b)-v(a)< n/2\}}|\Phi^0(g^{-1}(a+b\sqrt{D})g)|   de.\notag
\end{align}
For the piece of integral over $\F^*\backslash\{v(b)-v(a)\geq n/2\}$, we use the trivial bound  $|\Phi^0|\leq 1$ and 
\begin{equation}
\Vol(\F^*\backslash\{v(b)-v(a)\geq n/2\})\ll q^{-n/2}.
\end{equation}

For the piece of integral over $\F^*\backslash\{v(b)-v(a)< n/2\}$, we use the trivial bound 
\begin{equation}
\Vol(\F^*\backslash\{v(b)-v(a)< n/2\})\leq 1,
\end{equation}
and
\begin{equation}
\Phi^0(g^{-1}(a+b\sqrt{D})g)\ll q^{-\delta_0 n}.
\end{equation}
The latter inequality follows from (\ref{eq:decayofMC}) and 
\begin{equation}
v(\frac{a^2-b^2D}{b^2D^2\varpi^{2d}})=v(\frac{a^2-b^2D}{b^2D^2}\varpi^{2n})\geq n
\end{equation}
in the second part of Lemma \ref{lemofcosetdecomp}. Putting together, we have
\begin{equation}
|I(\Phi,\Omega)|\ll q^{-\delta_0 n}.
\end{equation}

Let's consider the case when $\E$ is split over $\F$ now. In this case, we fix an element $\sqrt{D}$ in the local field $\F$ and assume without loss of generality that $v(D)=0$. For 
\begin{equation}
\gamma=\zxz{1}{-\frac{1}{\sqrt{D}}}{\sqrt{D}}{1},
\end{equation}
we have
\begin{equation}
\gamma^{-1}\zxz{a}{b}{bD}{a} \gamma=\zxz{a+b\sqrt{D}}{0}{0}{a-b\sqrt{D}}.
\end{equation}
Let $u=a+b\sqrt{D}$ and $v=a-b\sqrt{D}$. The Haar measure on the split torus is $d^*ud^*v$. The volume of $\F^*\backslash\E^*$ is  not finite in this case. But as long as $a\nequiv \pm b\sqrt{D}$, $v(u)=v(v)$ and the total volume of such pieces is bounded by $1$. So we can apply exactly same argument as in the field extension case to control the integral on these pieces. 

We assume that $a\equiv \pm b\sqrt{D}$ now. We can assume without loss of generality that $v(a)=v(b)=0$. For any integer $j>0$, each piece $a\in \pm b\sqrt{D}+\varpi^jO_\F$ has volume $1$. But we will have more saving in the matrix coefficient as $j\rightarrow \infty$, as
\begin{equation}
v(\frac{a^2-b^2D}{b^2D^2}\varpi^{2n})\geq 2n+j,\text{\ \ } \Phi^0(g^{-1}(a+b\sqrt{D})g)\ll q^{-\delta_0 (2n+j)}.
\end{equation}
Then it's obvious that the sum over $j$ and $\pm$ signs of the integrals is still controlled by $q^{-\delta_0 n}$.
\end{proof}

\section{Test vectors and evaluation of Waldspurger's local integral on $\GL_2$ side for supercuspidal representations}
In this section we shall provide test vectors for a nontrivial element in $\Hom_{\E^*}(\pi\otimes\Omega,\C)$, in the setting where $\pi$ is supercuspidal with sufficiently large level compared to $\Omega$. But instead of working abstractly in representation theory, we will directly evaluate Waldspurger's local integral on candidates of test vectors. We will show that we can always find test vector for Waldspurger's local integral, as long as $\Hom_{\E^*}(\pi\otimes\Omega,\C)$ is not trivial.

We will first find a test vector of form $\charf_{\eta,d}$ in the Kirillov model.
Note that for supercuspidal representations, vectors of form $\charf_{\eta,d}$ provide a basis.  

The second reason for working with such test vectors is that they can be easily identified as local component of certain globally well-defined automorphic forms. The formula for their matrix coefficient is also easier.

The last reason is that it somewhat simplifies the process to search for test vectors. Note that Waldpurger's local integral gives an element in $\Hom_{\E^*}(\pi\otimes\Omega,\C)\otimes \Hom_{\E^*}(\hat{\pi}\otimes\Omega^{-1},\C)$. So to get a nonvanishing result for Waldspurger's local integral means we are finding test vectors simultaneously for both $\Hom_{\E^*}(\pi\otimes\Omega,\C)$ and $\Hom_{\E^*}(\hat{\pi}\otimes\Omega^{-1},\C)$. This seems a more difficult task, but we have the following lemma.
\begin{lem}
Assume that $w_\pi= \Omega|_{\F^*}=1$. If there exists test vectors $\charf_{\chi,n}, \charf_{\eta,m}\in \pi$ such that
\begin{equation}\label{eq:nonzerolocalint}
\int\limits_{\E^*}<\pi(e)\charf_{\chi,n}, \charf_{\eta,m}>\Omega(e)de\neq 0,
\end{equation}
then
\begin{equation}
J=\int\limits_{\E^*}<\pi(e)\charf_{\eta,m}, \charf_{\chi,n}>\Omega(e)de\neq 0
\end{equation}
\end{lem}

\begin{rem}
This lemma implies that if $\charf_{\chi,n}$ is a test vector for $\Hom(\pi\otimes \Omega, \C)$, it's automatically a test vector for $\Hom(\pi\otimes \Omega^{-1}, \C)$.
As a consequence, it would be enough to just test on matrix coefficient associated to the same vector, which is the case in Proposition \ref{propofsupportofMC} and \ref{propofMConTorus}. And if all matrix coefficients of this form fail, the Waldspurger's local integral will be trivial.
\end{rem}
\begin{proof}
\begin{align}
J&=\int\limits_{\E^*}<\charf_{\eta,m}, \pi(e^{-1})\charf_{\chi,n}>\Omega(e)de\\
&=\overline{\int\limits_{\E^*}<\pi(e)\charf_{\chi,n}, \charf_{\eta,m}>\Omega(e^{-1})de}\notag
\end{align}
By assumption, $\Omega(e)\Omega(\overline{e})=\Omega(N_{\E/\F}(e))=1$, so $\Omega(e^{-1})=\Omega(\overline{e})$. Then
\begin{equation}
J=\overline{\int\limits_{\E^*}<\pi(\overline{e})\charf_{\chi,n}, \charf_{\eta,m}>\Omega(e)de}.
\end{equation}
When $e=a+b\sqrt{D}=\zxz{a}{b}{bD}{a}$, $$\overline{e}=\zxz{a}{-b}{-bD}{a}=\zxz{-1}{0}{0}{1}e\zxz{-1}{0}{0}{1}.$$
Thus
\begin{equation}
J=\overline{\int\limits_{\E^*}<\pi(e)\pi(\zxz{-1}{0}{0}{1})\charf_{\chi,n}, \pi(\zxz{-1}{0}{0}{1})\charf_{\eta,m}>\Omega(e)de}.
\end{equation}
Now note that $\zxz{-1}{0}{0}{1}$ acts on $\charf_{\chi,n}$ or $\charf_{\eta,m}$ by a simple nonzero multiple. So the condition in the lemma directly implies the nonvanishing of $J$.
\end{proof}

\subsection{Matrix coefficient on torus}
We shall consider now the matrix coefficient for $\charf_{\eta,0}$. We will work with the case $c(\eta)\leq c(\pi)/2$ (which turns out to be enough).
Let $\Phi_\eta $ denote the matrix coefficient associated to $\charf_{\eta,0}$. It's related to the matrix coefficient of a newform from twisted representation by (\ref{formulatwistMC}). 

%\begin{rem}
%When we pick  $d\neq k$ and $c(\eta)\leq c(\pi)/2$, one can easily show that the local integral is zero for the test vector $\charf_{\eta,d}$ using Proposition \ref{propofsupportofMC}. For example when $i<c-1$, $\Phi^{(i)}(a,m)$ is supported at $v(m)=i-c$. Meanwhile in (\ref{eq:cosetdecomp}), $v(\frac{b}{a\varpi^d})=i-2d-v(D)\neq i-c$. So there will be no contributions from $i<c-1$.
%\end{rem}
We need to write down the value of $\Phi_\eta$ on conjugated torus more explicitly.
Because of the vanishing result in Proposition \ref{prop:vanishinglocalint}, we shall pick $d=k$.  According to Remark \ref{rem:twistforlocalint}, we can assume $w_\pi$ to be at most level $1$. For simplicity of notations, however, we will assume $w_\pi$ to be trivial in the following.
\begin{prop}\label{propofMConTorus}
Let $c(\pi)=2k$ or $2k+1$. Pick $d=k$. For $i=v(\frac{bD\varpi^k}{a})\geq c/2$,
\begin{align}
\Phi_\eta(\zxz{a}{b\varpi^{-k}}{bD\varpi^{k}}{a} )=\sum\limits_{\chi}\eta(\frac{a^2-b^2D}{a^2})\chi(\frac{b^2D}{a^2-b^2D})C_{\chi\eta^{-1}}C_\eta\int_{i-c}\psi^-\chi^{-1}\int_{i-c}\psi\chi^{-1}
\end{align}
where the sum is over level $c-i$ characters (and also level 0 character if $i=c-1$). 

For $i=v(\frac{bD\varpi^k}{a})\leq c/2$,
\begin{align}
\Phi_\eta(\zxz{a}{b\varpi^{-k}}{bD\varpi^{k}}{a} )
&=\sum_\chi C_{\eta\chi}\eta(\frac{a^2-b^2D}{a^2})\int_{-i}\psi\chi^{-1}\int\limits_{v(\alpha)=0}\psi(\varpi^{i-c}\frac{a^2}{a^2-b^2D}\alpha)\chi^{-1}\eta^{-2}(\alpha)d^*\alpha\\
&=\sum_\chi (\eta\chi)(\frac{a^2}{a^2-b^2D})C_{\eta\chi}\int_{-i}\psi\chi^{-1}\int_{-i}\psi\eta^{-2}\chi^{-1},
\end{align}
where the sum is over level $i$ characters (and also level 0 character if $i=1$).

\end{prop}
\begin{proof}
Suppose that $i\geq c/2$ first.
In general we can define the matrix coefficient via
\begin{equation}
\Phi_\eta(g)=\int\limits_{\F^*}\pi(g)\charf_{\eta,0}\overline{\charf_{\eta,0}}d^*x.
\end{equation}
Note that for $a\in O_F^*$,
\begin{equation}
\pi(\zxz{a}{0}{0}{1})\charf_{\eta,0}=\eta(a)\charf_{\eta,0}.
\end{equation}
Combining Lemma \ref{lemofcosetdecomp}, it would be enough to know the following
\begin{equation}
\Phi(\zxz{1}{m}{0}{1}\zxz{1}{0}{\varpi^i}{1})
=\int\limits_{v(x)=0}\psi(m x)\pi(\zxz{1}{0}{\varpi^i}{1})\charf_{\eta,0}(x)\overline{\eta}(x)d^*x.
\end{equation}
Then we just use that
\begin{equation}
\zxz{1}{0}{\varpi^i}{1}=-\zxz{0}{1}{-1}{0}\zxz{1}{-\varpi^i}{0}{1}\zxz{0}{1}{-1}{0},
\end{equation}
and compute the action of $\zxz{1}{0}{\varpi^i}{1}$ step by step. (And use Fourier expansion for the action of $\zxz{1}{-\varpi^i}{0}{1}$.) We shall skip the details here. 

When $i\leq c/2$, the computations are very similar, except that we shall now use
\begin{align}
\zxz{1}{0}{\varpi^i}{1}&=\zxz{\varpi^{-i}}{0}{0}{1}\zxz{1}{0}{1}{1}\zxz{\varpi^{i}}{0}{0}{1}\\
&=-\zxz{\varpi^{-i}}{0}{0}{1}\zxz{1}{1}{0}{1}\zxz{0}{1}{-1}{0}\zxz{1}{1}{0}{1}\zxz{\varpi^{i}}{0}{0}{1}.\notag
\end{align}
\end{proof}

%Recall that the constants $C_\chi$ is the formulae can be related to epsilon factors directly:
%\begin{equation}
%C_\chi=\epsilon(\pi\otimes\chi^{-1},\psi,1/2),
%\end{equation}
%which in turn can be written as a classical Gauss integral.
%

\subsection{inert extension}\label{secSCinert}

Note that there are now two quadratic extension appearing: the field $\E$ on which $\Omega $ is defined, and $\E'$ together with a character $\theta$ which defines $\pi$. When $\E$ is unramified and $c(\pi)$ is even, $\E'=\E$. We shall first give a lemma telling us what to expect for local integral. It follows directly from the calculation in \cite{Tu83}.
\begin{lem}\label{lem:epsiloninertsc}
Suppose that $\Omega$ is defined over an inert extension $\E$.
\begin{enumerate}
\item If $\pi$ is of level $2k+1$ and $c(\Omega)\leq k$, then $\epsilon(\Pi\otimes\Omega,1/2)=-1;$
\item If $\pi$ is of level $2k$ and $c(\Omega)<k$, then $\epsilon(\Pi\otimes\Omega,1/2)=1.$
\end{enumerate}
\end{lem}

Note that in this case $c(\pi_\Omega)=2c(\Omega)$. From now on we assume that $c(\pi)=2k$ is even.

\begin{prop}\label{proptestvecinertsc}
Assume that $2\nmid q$. Suppose that $\E$ is inert, $k\geq 2$, $c(\pi)=2k$ is even and $k>  c(\Omega)$. Then there exists a test vector of form $\charf_{\eta,-k}$ in the Kirillov model for $c(\eta)\leq k$ such that
\begin{equation}
I\geq \frac{1}{(q^2-1)q^{k-2}}.
\end{equation}

\end{prop}
Note that when $k=1$ and $\Omega$ is unramified, this is the case covered by Gross and Prasad's paper. Thus we assume $k\geq 2$.

One can easily do a weighted average for the test vector above and use Lemma \ref{lem:normalsubK} to get the following:
\begin{cor}\label{corinertsc}
With the same conditions as above, there exists a non-trivial test vector for any nontrivial element in $\Hom_{\E^*}(\pi\otimes\Omega,\C)$, such that it is invariant under  $K_1^1(\varpi^k,\varpi^k)$, and $\E^*$ acts on it by the character $\Omega^{-1}$.
\end{cor}

\begin{proof}
%Note that when $k=1$, $c(\pi)=2$ and $c(\Omega)=0$, this is the setting for the Gross-Prasad test vector. Thus we assume that $k>1$.

Recall in this setting we are picking $$g=  \zxz{\varpi^{k}}{0}{0}{1}$$ and $d=k$. $\Omega $ and $\theta$ are defined over the same inert extension $\E$. We shall pick $\varpi_\E=\varpi$ in this case.

To compute the integral, we use that both functions in the integral are invariant by $1+\varpi^kO_\E$. This is obvious for $\Omega$ by assumption on the levels. For $\Phi$, recall that $\Phi_\eta $ is invariant by $K_0(\varpi^{2k})$. So $\Phi$ is invariant by $$\zxz{\varpi^{k}}{0}{0}{1}K_0(\varpi^{2k})\zxz{\varpi^{-k}}{0}{0}{1},$$
which contains $1+\varpi^kO_\E$ under the standard embedding.

So
\begin{align}\label{formulaofwholesum}
I=&\int\limits_{O_\F^*\backslash O_\E^*}\Phi(e)\Omega(e)de\\
=&\frac{1}{(q+1)q^{k-1}}[\sum\limits_{b\in \varpi O_\F/\varpi^kO_\F}\Phi_\eta (g^{-1}(1+b\sqrt{D})g)\Omega(1+b\sqrt{D})\notag\\
&+\sum\limits_{a\in  O_\F/\varpi^kO_\F}\Phi_\eta (g^{-1}(a+\sqrt{D})g)\Omega(a+\sqrt{D})].\notag
\end{align}
We shall organize the sum according to the valuation of $a$ or $b$, and apply Lemma \ref{lemofcosetdecomp} and Proposition \ref{propofMConTorus} repeatedly.

\deleted{We first consider the case $k\geq 2$.}
When  $a=1$ and $v(b)>0$, we have $c/2<i=v(b)+k\leq c$ in Lemma \ref{lemofcosetdecomp}. In this case, 
\begin{equation}
\Phi_\eta (g^{-1}(1+b\sqrt{D})g)=\sum\limits_{\chi}\eta(1-b^2D)\chi(\frac{b^2D}{1-b^2D})C_{\chi\eta^{-1}}C_\eta\int_{i-c}\psi^-\chi^{-1}\int_{i-c}\psi\chi^{-1}.
\end{equation}
As functions in $b$ for fixed $v(b)$, $\eta(1-b^2D)$ is of smaller level $c(\eta)-2v(b)$, and $\chi(\frac{b^2D}{1-b^2D})$ is of level $c(\chi)=c-i=k-v(b)$, unless $\chi=\quadchar$. On the other hand $\Omega(1+b\sqrt{D})$ is also of smaller level $c(\Omega)-v(b)$. So the sum in $b$ would be zero unless $\chi$ is trivial or $\quadchar$. 

In particular the nonzero contributions will come from $\chi=1$ when $v(b)=k$, and $\chi=1$ or $\quadchar$ when $v(b)=k-1$. Both $\eta(1-b^2D)$ and $\Omega(1+b\sqrt{D})$ would be trivial for these pieces.
Then one can compute that
\begin{align}
I_1&=\sum\limits_{b\in \varpi O_\F/\varpi^kO_\F}\Phi_\eta (g^{-1}(1+b\sqrt{D})g)\Omega(1+b\sqrt{D})\\
&=1+(q-1)[\frac{1}{(q-1)^2}+\quadc{D}C_{\quadchar\eta^{-1}}C_\eta \frac{q}{(q-1)^2}] \notag\\
&=\frac{q}{q-1}[1-\frac{C_{\quadchar\eta^{-1}}}{C_{\eta^{-1}}}].\notag
\end{align}
Here we have used that $C_\eta C_{\eta^{-1}}=1$ and Lemma \ref{lemofGaussint}.

Now fix $b=1$ and $v(a)\geq 0$, and $0\leq i=k-v(a)\leq k$. This time
\begin{equation}
\Phi_\eta (g^{-1}(a+\sqrt{D})g)= \sum_\chi (\eta\chi)(\frac{a^2}{a^2-D})C_{\eta\chi}\int_{-i}\psi(\eta\chi)^{-1}\eta\int_{-i}\psi(\eta\chi)^{-1}\eta^{-1}
\end{equation}
The analysis is similar to the previous case. When $\eta =1$, the nonzero contributions will come from $v(a)=k$, $\chi=1$, and $v(a)=k-1$, $\chi=1$ or $\quadchar$. Then
\begin{align}
I_2&=\sum\limits_{a\in  O_\F/\varpi^kO_\F}\Phi_\eta (g^{-1}(a+\sqrt{D})g)\Omega(a+\sqrt{D})\\
&=\{C_1+(q-1)[C_1\frac{1}{(q-1)^2}+\quadc{-D}C_{\quadchar}\quadc{-1}\frac{q}{(q-1)^2}]\}\Omega(\sqrt{D}) \notag\\
&=\frac{q}{q-1}C_1\Omega(\sqrt{D})[1-\frac{C_{\quadchar}}{C_{1}}]. \notag
\end{align}
Note that both $\Omega(\sqrt{D})$ and $C_1$ are of values $\pm 1$.

When $ 0<c(\eta)< k-c(\Omega)$, the nonzero contributions will come from $v(a)=k-c(\eta)$ and $\chi=\eta^{-1}$ or $\eta^{-1}\quadchar$, since $\Omega$ will be trivial. Then
\begin{equation}
I_2=\frac{q}{(q-1)}\eta(-1)\Omega(\sqrt{D})C_1[1-\frac{C_{\quadchar}}{C_{1}}]
\end{equation}

When $ c(\eta)\geq k-c(\Omega)$, the nonzero contributions will still come from $v(a)=k-c(\eta)$ and $i=c(\eta)$ since this is the only chance $\chi$'s will have same level as $\eta$, thus $\chi\eta$ can be of smaller level, matching the level of $\Omega$. Then we shall write $\chi=\eta^{-1}\nu$, and only care about those $\nu$'s which are of level$\leq c(\Omega)-k+c(\eta)$. Then

\begin{equation}\label{formulaI2}
I_2=\sum\limits_{\nu \text{\ of level}\leq c(\Omega)-k+c(\eta)}\sum\limits_{v(a)=k-c(\eta)}\nu(\frac{a^2}{a^2-D})\Omega(a+\sqrt{D})C_\nu\int_{-c(\eta)}\psi\nu^{-1}\eta\int_{-c(\eta)}\psi\nu^{-1}\eta^{-1}.
\end{equation}

\deleted{The integral representation of the value of matrix coefficient is}

\deleted{$\int \psi\eta(\frac{a^2}{a^2-D}\alpha)\int \theta^{-1}(x)\psi\eta^{-1}(\frac{1}{\alpha}N(x))\psi_\E(x)d^*xd^*\alpha.$}

Now if we write $\eta=\eta_0\eta_1$ where $\eta_0$ is level 1, %as in Lemma \ref{lemofstructureofchar}.
we have
\begin{align}
&\int_{-c(\eta)}\psi\nu^{-1}\eta\cdot\int_{-c(\eta)}\psi\nu^{-1}\eta^{-1}\\
=&\eta_0(-\alpha_\eta)\int_{-c(\eta)}\psi\nu^{-1}\eta_1\cdot \eta_0^{-1}(\alpha_\eta)\int_{-c(\eta)}\psi\nu^{-1}\eta_1^{-1}\notag\\
=&\eta_0(-1)\int_{-c(\eta)}\psi\nu^{-1}\eta_1\cdot \int_{-c(\eta)}\psi\nu^{-1}\eta_1^{-1}.
\end{align}
Thus
\begin{equation}
I_2=\eta_0(-1)\sum\limits_{\nu \text{\ of level}\leq c(\Omega)-k+c(\eta)}\sum\limits_{v(a)=k-c(\eta)}\nu(\frac{a^2}{a^2-D})\Omega(a+\sqrt{D})C_\nu\int_{-c(\eta)}\psi\nu^{-1}\eta_1\int_{-c(\eta)}\psi\nu^{-1}\eta_1^{-1}.
\end{equation}
Note that while the sum in $a$ and $\nu$ is very difficult to evaluate in general, we can anyway change $\eta_0$ independently to get whatever sign we want for the contribution of $I_2$.

Now we use the imput from compact induction theory. In particular by Corollary \ref{corofquotientofC}, we have that

%\nu_\E(\frac{\alpha_\theta+\alpha_\eta \varpi_\E^{c(\theta)+e-1}\varpi^{-c(\eta)}}{\varpi_\E^{c(\theta)+e-1}})

\begin{equation}
\frac{C_{\quadchar\eta^{-1}}}{C_{\eta^{-1}}}=\quadc{N_{\E/\F}(\frac{\alpha_\theta+\alpha_\eta \varpi^{c(\theta)-c(\eta)}}{\varpi^{c(\theta)}})}
\end{equation}

Since $\theta|_{\F^*}=w_\pi=1$, we can pick $\alpha_\theta=\alpha\sqrt{D}$ for some $\alpha\in O_v^*$. Recall we extend the character $\quadchar$ such that $\quadc{\varpi}=1$. Then 

\begin{equation}
\frac{C_{\quadchar\eta^{-1}}}{C_{\eta^{-1}}}=\begin{cases}
\quadc{-\alpha^2D}=-\quadc{-1}, &\text{\ if } c(\eta)<c(\theta)\\
\quadc{\alpha_\eta^2-\alpha^2D}, &\text{\ if }c(\eta)=c(\theta)
\end{cases}
\end{equation}
So if $\quadc{-1}=1$, $I_1\neq 0$ for smaller level $\eta$, and we can in particular pick $\eta$ of level 1 such that $\eta(-1)C_1\Omega(\sqrt{D})=1$. Then the integral is
\begin{equation} \label{formulaExplicitlocalInt1}
I(\Phi,\Omega)=\frac{4}{(q^2-1)q^{k-2}}.
\end{equation}

Now if $\quadc{-1}=-1$, there exists $\alpha_\eta$ such that $\quadc{\alpha_\eta^2-\alpha^2D}=-1$. (This is because the norm map from the residue field $k_{\E^*}$ to $k_{\F^*}$ should send half points to nonsquares as the corresponding quadratic character is unramified.) So we can pick proper $\eta_1$ of level $k$ such that $I_1\neq 0$. Then we pick proper $\eta_0$ which doesn't affect $\alpha_\eta$, such that $I_2$ will not cancel $I_1$ and at least half of $I_1$ will be left for $I$.

In any case, we find a nontrivial test vector $\charf_{\eta,-k}$ for the local Waldspurger's period integral and get a lower bound for $I$.
\end{proof}
Note that we actually give explicit evaluation of local integral in (\ref{formulaExplicitlocalInt1}) when $\quadc{-1}=1$. Now we suppose that  $\quadc{-1}=-1$.
We shall evaulate $I$ more explicitly under stronger condition, that is, $k\geq 2c(\Omega)$.

\begin{prop}\label{propofexplicitI}
Assume that $2\nmid q$. Suppose that $\E$ is unramified, $c(\pi)=2k$ is even and $k\geq 2c(\Omega)$.
Further suppose that $\quadc{-1}=-1$. 
Then $I=0$ if $c(\eta)<k$, and there exists a proper level $k$ test vector  $\charf_{\eta,-k}$ such that the local integral
\begin{equation}
I=\frac{1}{(q^2-1)q^{k-2}}\{2+C_1\eta(-1)[\Omega(\sqrt{\frac{\alpha_\eta^2 D}{\alpha_\eta^2-\alpha_\theta^2D}}+\sqrt{D})+\Omega(-\sqrt{\frac{\alpha_\eta^2 D}{\alpha_\eta^2-\alpha_\theta^2D}}+\sqrt{D}]\}
\end{equation}
is nonvanishing and about size $\frac{1}{q^k}$.
\end{prop}
\begin{proof}
We just compute $I_2$ directly. Note that with the new condition, $\nu$ is of level $\leq c(\Omega)-k+c(\eta)\leq c(\eta)-\frac{k}{2}\leq \frac{c(\eta)}{2}\leq \frac{k}{2}$. Now we start with (\ref{formulaI2}) and apply Lemma \ref{lemoftwocharGaussint} and Corollary \ref{corofquotientofC}. We will get
\begin{equation}
I_2=\sum\limits_{\nu \text{\ of level}\leq c(\Omega)-k+c(\eta)}\sum\limits_{v(a)=k-c(\eta)}\nu(\frac{a^2}{a^2-D})\Omega(a+\sqrt{D})C_\nu\int_{-c(\eta)}\psi\nu^{-1}\eta\int_{-c(\eta)}\psi\nu^{-1}\eta^{-1}.
\end{equation}

\begin{equation}
\int_{-c(\eta)}\psi\nu^{-1}\eta\int_{-c(\eta)}\psi\nu^{-1}\eta^{-1}=\nu(-\alpha_\eta^{-2})\int_{-c(\eta)}\psi\eta\int_{-c(\eta)}\psi\eta^{-1}=\nu(-\alpha_\eta^{-2})\eta(-1)\frac{q}{(q-1)^2q^{c(\eta)-1}},
\end{equation}
and
\begin{equation}
C_\nu= C_1 \nu(-\alpha_\theta^2 D).
\end{equation}
So
\begin{equation}
I_2=\frac{q}{(q-1)^2q^{c(\eta)-1}}C_1\eta(-1)\sum\limits_{v(a)=k-c(\eta)}\Omega(a+\sqrt{D})\sum\limits_{\nu \text{\ of level}\leq c(\Omega)-k+c(\eta)}\nu(\frac{a^2\alpha_\eta^{-2}\alpha_\theta^2 D}{a^2-D}).
\end{equation}

Note that 
\begin{equation}
\sum\limits_{\nu \text{\ of level}\leq c(\Omega)-k+c(\eta)}\nu(x)
\end{equation}
is $(q-1)q^{c(\Omega)-k+c(\eta)-1}$ times the characteristic function of $1+\varpi^{c(\Omega)-k+c(\eta)}O_\F$.
Note that 
\begin{equation}
a^2\alpha_\eta^{-2}\alpha_\theta^2 D\equiv a^2-D
\end{equation}
has no solutions for fixed $v(a)$ as $\quadc{-1}=-1$, unless $v(a)=0$ and $c(\eta)=k$.
On the other hand when  $v(a)=0$ and $c(\eta)=k$, there exists $\alpha_\eta$ such that $\quadc{\alpha_\eta^2-\alpha_\theta^2D}=-1$, so there are exactly two solutions $\mod \varpi^{c(\Omega)-k+c(\eta)}O_\F$
\begin{equation}
a\equiv \pm\sqrt{\frac{\alpha_\eta^2 D}{\alpha_\eta^2-\alpha_\theta^2D}}.
\end{equation}
Recall that we are summing over $a\in O_\F/\varpi^k O_\F$, so we get
\begin{equation}
I_2=\frac{q}{(q-1)}C_1\eta(-1)[\Omega(\sqrt{\frac{\alpha_\eta^2 D}{\alpha_\eta^2-\alpha_\theta^2D}}+\sqrt{D})+\Omega(-\sqrt{\frac{\alpha_\eta^2 D}{\alpha_\eta^2-\alpha_\theta^2D}}+\sqrt{D}],
\end{equation}
and 
\begin{equation}
I=\frac{1}{(q^2-1)q^{k-2}}\{2+C_1\eta(-1)[\Omega(\sqrt{\frac{\alpha_\eta^2 D}{\alpha_\eta^2-\alpha_\theta^2D}}+\sqrt{D})+\Omega(-\sqrt{\frac{\alpha_\eta^2 D}{\alpha_\eta^2-\alpha_\theta^2D}}+\sqrt{D}]\}.
\end{equation}
Now we can pick $\eta=\eta_0\eta_1$ properly such that $I$ in not zero and about size $\frac{1}{q^{k}}$.

\end{proof}

%\begin{rem}
%In particular we can pick $\eta'$ which gives the same $\alpha_\eta$ and opposite $\eta(-1)$, and 
%\begin{equation}
%I+I'=\frac{4}{(q^2-1)q^{k-2}}.
%\end{equation}
%\end{rem}

%\section{Proof of the Main theorem}

\subsection{ramified extension}
Now we assume $\E$ to be a ramified extension over $\F$. For simplicity we suppose that the uniformizer of $\F$ is chosen such that $\E=\F(\sqrt{\varpi})$ (so $D=\varpi$ in this case), and the uniformizer for $\E$ is chosen to be $\varpi_\E=\sqrt{D}$. 
 The calculation in this case is similar to the inert case, so we shall mainly focus on the differences.

We first reformulate the results in \cite{Tu83} on the epsilon factor.

\begin{lem}\label{lem:epsilonramifiedsc}
\begin{enumerate}
\item Suppose that $\pi$ is of level $2k$, and $\Omega$ is a character over ramified extension $\E$ such that $c_\E(\Omega)\leq 2k-1$, then
$\epsilon(\Pi\otimes\Omega,1/2)=-1$;
\item Suppose that $\pi$ is of level $2k+1$, associated to a character $\theta $ over another ramified quadratic extension $\E'$ (could be different from $\E$) such that $c_\E(\Omega)<c_{\E'}(\theta)=2k$. Let $\xi$ be a unit such that $\varpi_{\E'}^2=\xi\varpi$.
 Then $\epsilon(\Pi\otimes\Omega,1/2)=1$ if and only if $\quadc{-\xi}=1.$
\end{enumerate}
\end{lem}

\begin{rem}
When $\pi$ is odd level, we have combined Proposition 2.9 and part (a) of Proposition 2.10 of \cite{Tu83}, together with the simplifying assumption that $c_\E(\Omega)<c_{\E'}(\theta)$.
\end{rem}

\begin{prop}\label{prop:testvecramifiedsc}
Assume that $2\nmid q$.
Suppose that $\pi$ is a supercuspidal representation of level $2k+1$, associated to a character $\theta $ over another ramified quadratic extension $\E'$ with $\varpi_{\E'}^2=\xi\varpi$, such that $\quadc{-\xi}=1.$. Then there exists a proper test vector $\charf_{\eta,-k}$ where $\eta$ is of level at most 1, such that the associated local integral
\begin{equation}
I(\Phi,\Omega)=\frac{2}{(q-1)q^{k-1}}..
\end{equation}
\end{prop}

\begin{cor}\label{cor:ramifiedsc}
With the same conditions as above, there exists a non-trivial test vector for any nontrivial element in $\Hom_{\E^*}(\pi\otimes\Omega,\C)$, such that it is invariant under  $K_1^1(\varpi^{k+1},\varpi^k)$, and $\E^*$ acts on it by the character $\Omega^{-1}$.
\end{cor}

\begin{proof}
In this case we shall pick $d=k$ for Lemma \ref{lemofcosetdecomp}, and $\eta $ is of level at most $1$.
Lemma \ref{lemofcosetdecomp} and Proposition \ref{propofMConTorus} still hold while we should keep in mind that $v(D)=1$. Then $\Phi_\eta$ would be invariant by $1+\varpi_\E^{2k} O_\E$, and so is $\Omega$ by the assumption on its level. Then as in the inert case, we can evaluate the local integral by the following finite sum:
\begin{align}
I=&\int\limits_{\F^*\backslash \E^*}\Phi_\eta(e)\Omega(e)de\\
=&\frac{1}{2q^{k}}[\sum\limits_{b\in  O_\F/\varpi^kO_\F}\Phi_\eta (g^{-1}(1+b\sqrt{\varpi})g)\Omega(1+b\sqrt{\varpi})\notag\\
&+\sum\limits_{a\in  \varpi O_\F/\varpi^{k+1}O_\F}\Phi_\eta (g^{-1}(a+\sqrt{\varpi})g)\Omega(a+\sqrt{\varpi})].\notag
\end{align}
Here we have chosen the normalization such that $vol(O_\E^*)=1$.
Similarly as in the inert case, only a few terms will have contribution to the whole integral due to the assumption on the levels. In particular
\begin{equation}
I=\frac{1}{2q^{k}}(I_1+I_2),
\end{equation}
where
\begin{equation}
I_1=1+(q-1)[\frac{1}{(q-1)^2}+\quadc{D}C_{\quadchar\eta^{-1}}C_\eta \frac{q}{(q-1)^2}] =\frac{q}{q-1}[1+\frac{C_{\quadchar\eta^{-1}}}{C_{\eta^{-1}}}],
\end{equation}
and 
\begin{equation}
I_2%=\{C_1+(q-1)[C_1\frac{1}{(q-1)^2}+\quadc{-D}C_{\quadchar}\quadc{-1}\frac{q}{(q-1)^2}]\}\Omega(\sqrt{D}) 
=\frac{q}{q-1}C_1\eta(-1)\Omega(\sqrt{D})[1+\frac{C_{\quadchar}}{C_{1}}]
\end{equation}
for $\eta$ at most level $1$.
Here we have used that $D=\varpi$ and $\quadc{\varpi}=1$ by our convention to extend characters of $O_\F^*$ to be characters of $\F^*$.

Now we use the results from compact induction theory. Suppose $\pi=\pi_\theta$ where $\theta$ is a character over a possibly different ramified quadratic extension $\E'$. Let $\varpi_{\E'}$ be the local uniformizer for $\E'$. Then using Corollary \ref{corofquotientofC}, we have

\begin{equation}
\frac{C_{\quadchar\eta^{-1}}}{C_{\eta^{-1}}}=\quadc{N_{\E'/\F}(\frac{\alpha_\theta+\alpha_\eta \varpi_{\E'}^{c(\theta)+e-1}\varpi^{-c(\eta)}}{\varpi_{\E'}^{c(\theta)+e-1}})}.
\end{equation}
Since $2c(\eta)\leq 2<c(\theta)+e-1$ and $\quadchar$ is of level 1, we can ignore the contribution from $\alpha_\eta$.
Since $\theta|_{\F^*}=w_\pi=1$, and
$$\theta(1+x)=\psi_{\E'}(\frac{\alpha_\theta}{\varpi_{\E'}^{c(\theta)+1}}x) $$
 we can pick $\alpha_\theta\in O_v^*$. Recall $c(\theta)=2k$ when $c(\pi_\theta)=2k+1$. Then
\begin{equation}
\frac{C_{\quadchar\eta^{-1}}}{C_{\eta^{-1}}}=\quadc{\frac{\alpha^2}{(-1)^{2k+1}\varpi_{\E'}^{4k+2} }}=\quadc{-(\xi\varpi)^{2k+1} }=\quadc{-\xi}.
\end{equation}

%
%\begin{equation}
%\frac{C_{\quadchar\eta^{-1}}}{C_{\eta^{-1}}}=\quadc{\varpi_{\E'}^{-4c(\eta)}\alpha_\eta'^{2}-\varpi_{\E'}^{-4k-2}\alpha_\theta^2}=\quadc{-\varpi_{\E'}^{2}}=\quadc{-\xi}.
%\end{equation}
So by assumption on $\quadc{-\xi}$, we have $\frac{C_{\quadchar\eta^{-1}}}{C_{\eta^{-1}}}=1$.
Now we choos $\eta$ properly such that $C_1\eta(-1)\Omega(\sqrt{D})=1$. Then we have
\begin{equation}
I=\frac{2}{(q-1)q^{k-1}}.
\end{equation}
\end{proof}

\begin{rem}
An interesting difference for the ramified extension here is that we can't change the quotient $\frac{C_{\quadchar\eta^{-1}}}{C_{\eta^{-1}}}$ by choosing different $\eta$ as in the unramified case. But then the condition for the nonvanishing of local integral match exactly the condition on epsilon factor.
\end{rem}

\section{Test vector in case of principal sereis representation}
From now on we assume that $\pi$ is a principal sereis representation. From (\ref{eq:twistedlocalint}) we can assume $\pi$ to be minimal. In particular suppose without loss of generality that $\pi=\pi(1,\mu)$. Note that $\Hom_{\E^*}(\pi\otimes\Omega,\C)\neq 0$ in this case so we always expect to find a test vector. We shall evaluate Waldspurger's local integral directly just like supercuspidal representation case.

First recall from \cite{YH14} that if $\pi$ is of form $\pi(\mu_1,\mu_2)$, where $\mu_1$ is unramified and $\mu_2$ is ramified of level $n$. Then the level of $\pi$ is $n$. In this case the new form is right $K_0(\varpi^n)-$invariant and supported on $BK_1(\varpi^n)$.
Then we have the following result on the Whittaker functional associated to the newform:
\begin{lem}\label{Wiofbiasram}
\begin{enumerate}
\item[(1)]When $i=n$, 
\begin{align}
W^{(i)}(\alpha)=W(\zxz{\alpha}{0}{0}{1}\zxz{1}{0}{\varpi^i}{1})&=\int\limits_{v(m)\leq v(\alpha)-k}\mu_1(-\frac{\alpha}{m})\mu_2(-m)\psi(-m)q^{-\frac{1}{2}v(\alpha)+v(m)}dm\\
&=\begin{cases}
q^{-\frac{1}{2}v(\alpha)-n}\mu_1^{v(\alpha)+n}(\varpi)\int\limits_{v(m)=-k}\mu_2(-m)\psi(-m)dm, &\text{\ if\ }v(\alpha)\geq 0,\\
0,& \text{\ otherwise.}
\end{cases}\notag
\end{align}
\item[(2)]When $i<n$, 
\begin{equation}\label{formulaWibias}
W^{(i)}(\alpha)=\mu_1^i(\varpi)\int\limits_{u\in O_\F}\mu_2(\alpha\varpi^{-i}(1-\varpi^{n-i}u))\psi(\alpha\varpi^{-i}(1-\varpi^{n-i}u))q^{-\frac{1}{2}v(\alpha)-n+i}du.
\end{equation}
\end{enumerate}
In particular
\begin{equation}
W^{(0)}(\alpha)=\begin{cases}
\mu_2(\alpha)\psi(\alpha)q^{-\frac{1}{2}v(\alpha)-n}, &\text{\ if }v(\alpha)\geq -n\\
0, &\text{otherwise.}
\end{cases}
\end{equation}
\end{lem}
Note that we haven't normalize the Whittaker functional in this Lemma. One can compute associated matrix coefficient $\Phi^0$ using the Whittaker functional $W$ by
\begin{equation}
\Phi^0(\zxz{\alpha}{m}{0}{1}\zxz{1}{0}{\varpi^i}{1})=\int\limits_{v(x)\geq 0}\psi(mx)W^{(i)}(\alpha x)W^{(n)}(x)d^*x.
\end{equation}
\begin{cor}\label{cor:MCofinducedbias}
$\Phi^0(\zxz{\alpha}{m}{0}{1}\zxz{1}{0}{\varpi^i}{1})$ is supported at $v(\alpha)\leq i-n$ for $0<i<n$. 
When $i=0$ and $v(m)\geq \min\{0, v(\alpha)+n\}$
\begin{equation}
\Phi^0(\zxz{\alpha}{m-\alpha}{0}{1}\zxz{1}{0}{1}{1})=0.
\end{equation}
\end{cor}
\begin{proof}
The first statement follows from that $W^{(i)}(\alpha)$ for $0<i<n$ is supported only at $v(\alpha)=i-n$ for (\ref{formulaWibias}) to be nonzero. For the second statement, we just follow definition
\begin{align}
\Phi^0(\zxz{\alpha}{m-\alpha}{0}{1}\zxz{1}{0}{1}{1})&=\int\limits_{v(x)\geq 0}\psi((m-\alpha)x)W^{(0)}(\alpha x)W^{(n)}(x)d^*x\\
&=\int\limits_{v(x)\geq \max\{0, -v(\alpha)-n\}}\psi((m-\alpha)x)\mu_2(\alpha x)\psi(\alpha x)q^{-\frac{1}{2}v(\alpha x)-n}W^{(n)}(x)d^*x\notag\\
&=\int\limits_{v(x)\geq \max\{0, -v(\alpha)-n\}}\psi(mx)\mu_2(\alpha x)q^{-\frac{1}{2}v(\alpha x)-n}W^{(n)}(x)d^*x.\notag
\end{align}
As $v(m)\geq \min\{0, v(\alpha)+n\}$, $\psi(mx)=1$. Then the integral is zero as $\mu_2(\alpha x)$ is level $n$ in $x$.
\end{proof}
\subsection{$\Omega$ over inert field extension}

\begin{prop}\label{prop:inducedrepresentation}
Suppose that $\pi=\pi(1,\mu)$ for $\mu$ of level $2k$ or $2k+1$. Let $\varphi^0$ be the newform in $\pi$, $\Phi^0$ be the matrix coefficent associated to $\varphi^0$ and $\Phi$ be the matrix coefficient associated to $\pi(\zxz{\varpi^{k}}{0}{0}{1})\varphi^0$. 
Let $\Omega $ be a character over  \added{an inert} quadratic field extension $\E$\deleted{ with ramification index $e$}, such that $v(D)=0$ and $c(\Omega)\leq k$ or $k+1$. \deleted{level assumption}

Then
\begin{equation}
I(\Phi,\Omega)=\frac{1}{(q+1)q^{k-1}} \text{\ or }\frac{1}{(q+1)q^{k}}.
\end{equation}
\end{prop}

\begin{cor}\label{cor:inertInd}
With the same conditions as above, there exists a non-trivial test vector for any nontrivial element in $\Hom_{\E^*}(\pi\otimes\Omega,\C)$, such that it is invariant under  $K_1^1(\varpi^k,\varpi^k)$ or $K_1^1(\varpi^{k+1},\varpi^{k+1})$, and $\E^*$ acts on it by the character $\Omega^{-1}$.
\end{cor}

\begin{proof}
$c(\pi)=2k$ or $2k+1$.
\begin{equation}
I(\Phi,\Omega)=\int\limits_{\F^*\backslash\{v(\frac{bD\varpi^d}{a})\geq c(\pi) \}}\Phi(e)\Omega(e)de+\int\limits_{\F^*\backslash\{v(\frac{bD\varpi^d}{a}< c(\pi) \}}\Phi(e)\Omega(e)de.
\end{equation}
When $v(\frac{bD\varpi^d}{a})\geq c(\pi) $, 
\begin{equation}
v(\frac{a^2-b^2D}{abD}\varpi^{i-d})=v(a^2-b^2D)-2v(a)=0, \text{\ }v(\frac{b}{a\varpi^d})\geq c(\pi)-2d-v(D)\geq 0
\end{equation}
so $\Phi(e)\Omega(e)=1$ by the first part of Lemma \ref{lemofcosetdecomp} and the assumption on $c(\Omega)$. One can easily check that
\begin{equation}
\Vol(\F^*\backslash\{v(\frac{bD\varpi^d}{a})\geq c(\pi) \})= \frac{1}{(q+1)q^{k-1}} \text{\ or }\frac{1}{(q+1)q^{k}}.
\end{equation}
Now we show that the other parts of conjugated torus will completely miss the support of matrix coefficient.

When $0<i=v(\frac{bD\varpi^d}{a})<c(\pi)$, we have by Lemma \ref{lemofcosetdecomp}
\begin{equation}\label{eq:conjtorusposition1}
v(\frac{a^2-b^2D}{abD}\varpi^{i-d})=v(a^2-b^2D)-2v(a)=\begin{cases}
0,&\text{\  if }v(a)\leq v(b)\\
2v(b)+v(D)-2v(a), &\text{\ otherwise.}
\end{cases}
\end{equation}
On the other hand by Corollary \ref{cor:MCofinducedbias} the matrix coefficient $\Phi^0$ is supported at
\begin{equation}
v(\alpha)\leq i-c(\pi)=v(\frac{bD\varpi^d}{a})-c(\pi).
\end{equation} 
One can check from (\ref{eq:conjtorusposition1}) that
\begin{equation}
2v(b)+v(D)-2v(a)>v(\frac{bD\varpi^d}{a})-c(\pi)
\end{equation}
when $v(\frac{bD\varpi^d}{a})>0$. So these parts of conjugated torus miss the support.

When $v(\frac{bD\varpi^d}{a})\leq 0$, $i=0$. By the second part of Lemma \ref{lemofcosetdecomp}, we set
\begin{equation}
\alpha=\frac{a^2-b^2D}{b^2D^2\varpi^{2d}},\text{\ }m=\frac{a}{bD\varpi^d}
\end{equation}
for notations as in Corollary \ref{cor:MCofinducedbias} and easily check that 
\begin{equation}
v(m)=v(a)-v(b)-v(D)-d\geq 0.
\end{equation}
So the value of the matrix coefficient is zero on this piece according to Corollary \ref{cor:MCofinducedbias}.
\end{proof}

\subsection{$\Omega$ over ramified field extension}
By almost the same computation, we have the following results

\begin{prop}\label{propIindramified}
Suppose that $\pi=\pi(1,\mu)$ for $\mu$ of level $2k$ or $2k-1$. Let $\varphi^0$ be the newform in $\pi$, $\Phi^0$ be the matrix coefficent associated to $\varphi^0$ and $\Phi$ be the matrix coefficient associated to $\pi(\zxz{\varpi^{k-1}}{0}{0}{1})\varphi^0$. 
Let $\Omega $ be a character over a ramified quadratic field extension $\E$, such that $v(D)=1$ and $c(\Omega)\leq 2c(\pi)-2k+1$. 

Then
\begin{equation}
I(\Phi,\Omega)=\frac{1}{2q^{c(\pi)-k}}.
\end{equation}
\end{prop}

\begin{cor}\label{cor:ramifiedInd}
With the same conditions as above, there exists a non-trivial test vector for any nontrivial element in $\Hom_{\E^*}(\pi\otimes\Omega,\C)$, such that it is invariant under  $K_1^1(\varpi^{k+1},\varpi^k)$ or $K_1^1(\varpi^{k},\varpi^{k-1})$, and $\E^*$ acts on it by the character $\Omega^{-1}$.
\end{cor}

\begin{rem}\label{rem:principaltestvector}
Essentially same calculations hold when we pick smaller $d$. In that case we will get smaller value for the local integral while we can allow larger $c(\Omega)$. In particular we can provide test vector in this way for arbitrary setting of $c(\pi)$ and $c(\Omega)$ in case of principal series representation.
\end{rem}

\section{Application to mass equidistribution on nonsplit torus}
In this section we give a quick application of the results on local integral to mass equidistribution of a family of cusp forms on nonsplit torus. We shall go back to global notations.
\begin{theo}\label{them:massequi}
Let $f$ be an automorphic unitary cuspidal newform of finite conductor $N=q^c$ on $\GL_2$, with $L^2$ norm being $1$ and bounded archimedean components. Let $\E^*$ be a fixed nonsplit torus of $\GL_2$. \deleted{We assume a technical condition that locally $\E$ is not split at $p$.}  Then the mass measure associated to $f$ is equidistributed on $\E^*\backslash \A_\E^*$ as $c\rightarrow \infty$.
%$N\rightarrow \infty$.
\end{theo}
\begin{proof}
Let $\Omega$ be an fixed eigen test function (that is, a character) on $\E^*\backslash \A_\E^*$. Then the proof of the theorem amounts to show that
\begin{equation}
\int\limits_{[\E^*]}|f|^2(e)\Omega(e)de\rightarrow \int\limits_{[\E^*]}\Omega(e)de.
\end{equation}
We do a spectrum decomposition for $|f|^2$ and get
\begin{align}
&\int\limits_{[\E^*]}|f|^2(e)\Omega(e)de\\
=&\sum\limits_{\text{cusp forms $\varphi$}}<|f|^2,\varphi> \int\limits_{[\E^*]}\varphi(e)\Omega(e)de+<|f|^2,1> \int\limits_{[\E^*]}\Omega(e)de+\int\limits_{E}<|f|^2,E> \int\limits_{[\E^*]}E(e)\Omega(e)de.\notag
\end{align}
Note that $\varphi$ and $E$ must have trivial central characters.
The main term is the constant term 
$$<|f|^2,1> \int\limits_{[\E^*]}\Omega(e)de=\int\limits_{[\E^*]}\Omega(e)de$$
by normalization, which is exactly what we want. So we need to prove power saving in level aspect for both discrete spectrum and continuous spectrum. 
We only show the proof for discrete spectrum, while the proof for continuous spectrum is similar and easier. 

%We first consider the scenario where $N=q^c$ for $c\rightarrow \infty$.
%For simplicity we assume that $N=q^c$.
We can further organize the sum in $\varphi$ as follows
\begin{equation}
T_1=\sum\limits_{\pi, c(\pi)\leq c}\sum\limits_{\varphi\in \mathcal{B}(\pi,c)}<|f|^2,\varphi>\int\limits_{[\E^*]}\varphi(e)\Omega(e)de.
\end{equation}
Here %$C(\pi)$ is the finite conductor of $\pi$. 
$\mathcal{B}(\pi,N)$ is a basis of elements of $\pi$ of level up to $c$. In particular we can pick the basis to consist of a newform $\varphi^0$ and diagonal translates $\pi(\zxz{\varpi^{-n}}{0}{0}{1})\varphi^0$ for $ 0< n\leq c-c(\pi)$. 

Using Proposition \ref{prop:vanishinglocalint} and Remark \ref{rem:vanishinglocalint}, we can first reduce the sum in $\pi$ to those such that $c(\pi)\ll c(\Omega)$, since when $c(\pi)$ is large enough, the local integral is zero for test vectors in $\mathcal{B}(\pi,N)$. As a result, the sum in $\pi$ is short and it suffices to prove power saving in the sum over $\varphi$ for each fixed $\pi$.

Note that the sum in $\varphi$ has length at most $c+1\ll q^{c\epsilon}$. For each individual term, the power saving comes either from the triple product period integral $<|f|^2,\varphi>$ or Waldspurger's period integral $\int\limits_{[\E^*]}\varphi(e)\Omega(e)de$.

 In particular if $\varphi=\pi(\zxz{\varpi^{-n}}{0}{0}{1})\varphi^0$ for $n$ large enough, we can get power saving for $\int\limits_{[\E^*]}\varphi(e)\Omega(e)de$  by using the convexity bound for the L-functions in  Theorem \ref{WaldsMCversion} and power saving or vanishing results for the local integrals by Proposition \ref{prop:decaylocalint} and \ref{prop:vanishinglocalint}.
 
On the other hand if $\varphi=\pi(\zxz{\varpi^{-n}}{0}{0}{1})\varphi^0$, for $n$ small, we get power saving for $<|f|^2,\varphi>$ using the work in \cite{HMN16}.

\end{proof}

\end{document}